\documentclass[twoside]{article}

\usepackage[utf8]{inputenc}
\usepackage[T1]{fontenc}
\usepackage[english]{babel}

\usepackage{mathpazo}
\usepackage{mathptmx}
\usepackage{amsmath}
\usepackage{amssymb}
\usepackage{mathrsfs}
\usepackage{amsthm}
\usepackage{amsfonts}
\usepackage{fancyhdr}
\usepackage[all]{xy}

\newcommand{\fonction}[4]{\begin{array}{lrcl}
 & #1 & \longrightarrow & #2 \\
    & #3 & \longmapsto & #4\end{array}}

\usepackage[left=3cm, right=3cm]{geometry}

\title{Invariant measure and long time behavior of regular solutions of the Benjamin-Ono equation}
\author{ Mouhamadou Sy\\
  \scriptsize{Laboratoire AGM}\\
  \scriptsize{Université de Cergy-Pontoise, France}\\ \scriptsize{mouhamadou.sy@u-cergy.fr}}
\date{}

\usepackage{amsfonts, amssymb, amsmath, amsthm, multicol,bbm}
\usepackage[colorlinks=true, pdfstartview=FitV, linkcolor=blue,
            citecolor=blue, urlcolor=blue]{hyperref}
\usepackage[usenames]{color}
\definecolor{Red}{rgb}{0.7,0,0.1}
\definecolor{Green}{rgb}{0,0.7,0}
\usepackage{accents}
\usepackage{comment,url}

\definecolor{labelkey}{rgb}{0,0,1}

\numberwithin{equation}{section}

\newtheorem*{Thm*}{Theorem}

\usepackage{cjhebrew}

\newcommand{\E}{\mathbb{E}}

\newcommand{\T}{T}

\usepackage[numbers,sort&compress]{natbib}

{\theoremstyle {definition} \newtheorem {defi} {Definition} [section] }
{\theoremstyle {plain}  \newtheorem {thm} [defi] {Theorem}}
{\theoremstyle {plain}  \newtheorem {cor} [defi]{Corollary}}
{\theoremstyle {plain} \newtheorem {prop} [defi]{Proposition}}
{\theoremstyle {plain} \newtheorem {nem}[defi] {Lemma}}
{\theoremstyle {plain} }
{\theoremstyle {plain} \newtheorem {rmq}[defi] {Remark}}

{\theoremstyle {plain}  }
{\theoremstyle {plain}  }
{\theoremstyle {plain}  }
{\theoremstyle {plain} }
{\theoremstyle {plain} }
{\theoremstyle {plain} }
{\theoremstyle {plain} }

\newcommand{\osgn}{\mathop{\rm sgn}\limits}

\def\E{{\Bbb{E}}}
\def\T{{\Bbb{T}}}
\def\P{{\Bbb{P}}}
\def\R{{\Bbb{R}}}

\def\Z{{\Bbb{Z}_0}}
\def\N{{\Bbb{N}}}

\def\dx{{\partial_x}}

\def\i{{\text{i}}}
\def\dy{{\partial_x^2}}

\def\dt{{\partial_t}}
\def\du{{\partial_u}}
\def\f{{\mathfrak{F}}}

\def\f{{\mathfrak{F}}}

\def\lleq{{\ \lesssim\ }}

\begin{document}

\markboth{ M. Sy}
{Invariant measure for the Benjamin-Ono equation}

\maketitle

\begin{abstract}
The Benjamin-Ono equation describes the propagation of internal waves in a stratified fluid. In the present work, we study large time dynamics of its regular solutions via some probabilistic point of view. We prove the existence of an invariant measure concentrated on $C^\infty(\T)$ and establish some qualitative properties of this measure. We then deduce a recurrence property of regular solutions and other corollaries using ergodic theorems. The approach used in this paper applies to other equations with infinitely many conservation laws, such as the KdV and cubic Schrödinger equations in 1D. It uses the fluctuation-dissipation-limit approach and relies on a \textit{uniform} smoothing lemma for stationary solutions to the damped-driven Benjamin-Ono equation.


\end{abstract}

\nocite{*}

\section{Introduction}

\subsection*{The problem and statement of the main result}

The Benjamin-Ono (BO) equation 
\begin{equation}\label{BO}
\dt u+H\dy u+u\dx u=0
\end{equation}
describes the propagation of internal waves in a stratified fluid \cite{benjamin1967internal,ono}. The operator $H$ in the equation is the Hilbert transform, it can be defined in Fourier setting as the multiplier by $-\i\osgn$ (see Appendix). We assume that $u(t,x)$ is a real-valued function, $t\in\R_+$ and $x$ belongs to the torus $\T=\R/2\pi\Bbb Z.$ In this setting, existence and uniqueness of solution hold in any Sobolev space $H^s$ for $s\geq 0$ (see \cite{molinet2008global, molinet2012cauchy} for its global wellposedness in $L^2(\T)$, and references therein for other works on the Cauchy problem of the equation). In the present paper, we use only the wellposedness of the problem in Sobolev spaces $H^s(\T)$ with $s\geq 2$, so we refer the reader to \cite{abdelouhab1989nonlocal}.

In $L^2:=L^2(\T)$, the wellposedness of $(\ref{BO})$ generates a topological dynamical system (DS) $(L^2,\phi_t)$, where $\phi_t$ is the flow of the equation. We are concerned with the description of the long time behavior of this dynamical system.

Given a Borel measure $\mu$ on $L^2$, we say that $\mu$ is invariant for $(L^2,\phi_t)$ if for any Borel set $A$ of $L^2,$ we have
\begin{equation*}
\mu(\phi_t^{-1}A)=\mu(A),\ \ \ \text{ for all}\ t.
\end{equation*}    
When such a measure exists, the triple $(L^2,\phi_t,\mu)$ is called a measurable dynamical system (MDS). If in addition $\mu$ is finite, then we have very important information on the dynamics. Indeed  the Poincaré recurrence theorem states that the dynamics is recurrent, that is, $\mu$-almost every orbit returns in any neighborhood of its origin  in finite time.
 The well-known Von Neumann and Birkhoff ergodic theorems also apply to give more information on the long time behavior of the system. 
Our aim here is to construct such a measure, this will contribute to improving the understanding of the behavior of the solutions of $(\ref{BO})$.    

 Matsuno \cite{matsuno1984bilinear} derived (at least formally) infinitely many conservation laws for the BO equation $(\ref{BO}).$ They have the form
\begin{equation}\label{form_energies}
E_n(u)=\|u\|_n^2+R_n(u), \ \ \ \ n\in\frac{\N}{2},
\end{equation}
where $\|.\|_n$ stands for the homogeneous Sobolev norm of order $n$ and $R_n$ is a lower order term.

In \cite{tv2011gaussian,tzvetkov-visc,tv2014invariant,yudeng,dtv2014invariant} Deng, Tzvetkov and Visciglia constructed a sequence of invariant Gaussian type measures $\{\mu_n\}$ for $(L^2,\phi_t)$ satisfying
\begin{itemize}
\item $\mu_n$ is concentrated on $H^s(\T)$, for $s<n-\frac{1}{2},$ $\ \ \ \ \ \ \ \ \ \ \ (*)$
\item $\mu_n(H^{n-1/2}(\T))=0.$  $\ \ \ \ \ \ \ \ \ \ \ \ \ \ \ \ \ \ \ \ \ \ \ \ \ \ \ \ \ \ \ \ \ \ \ \ \ \ \ \ \ \ \ \ \ \  (**)$
\end{itemize}

Formally, $\mu_n$ is defined as a renormalization of
\begin{equation*}
d\mu_n(u)=e^{-E_n(u)}du=e^{-R_n(u)}e^{-\|u\|_n^2}du,
\end{equation*}
where $E_n(u)$ and $R_n(u)$ are the quantities given in $(\ref{form_energies}).$
These authors constructed a Gaussian interpretation of the expression $e^{-\|u\|_n^2}du$ on the concerned spaces and proved that $e^{-R_n(u)}$ is an integrable density. 
In view of these results, there is an MDS for $(\ref{BO})$ in any Sobolev space and then its large time dynamics is described keeping in mind the theorems mentioned above. However, these results do not apply to infinitely smooth solutions, indeed by the property $(**)$ we have
\begin{equation*}
\mu_n(C^\infty(\T))=0,\ \ \ \text{ for all $n$}.
\end{equation*}

In the present work, we construct a measurable dynamical system for $(\ref{BO})$ on the space $C^\infty(\T).$ Naturally, the Dirac measure at $0$ is not the desired measure although it is invariant under the flow of the BO equation, but it gives only trivial information. More generally, to get substantial information on the system we have to also avoid singular measures. Another example of such measures is the one concentrated on a stationary solution. Notice that measures $\mu_n$ discussed above verify the following "consistency" property: every set of full $\mu_n$-measure is dense in $\dot{H}^{(n-1/2)^-}.$ 
 Concerning the space $C^\infty$, an obstruction to the construction of an invariant Gaussian type measure is the non-existence of a conservation law compatible with the regularity of that space. In particular, the approach used in the construction of the measures $\mu_n$ above does not seem to apply.
 
Another method allowing the construction of invariant measures (a priori not of Gaussian type) for PDEs was developed in \cite{kuk_eul_lim, KS04} in the context of Euler and Schrödinger equations, respectively. It is based on a fluctuation-dissipation (FD) argument and consists of adding to the equation appropriately normalized damping and stochastic terms, constructing an invariant measure for the resulting problem, and passing to the limit. But, a priori, the obstruction encountered in the Gaussian type measure approach still remains in the FD approach because the underlying regularization is of Sobolev order and not $C^\infty$. The idea in the present work is to exploit the regularization inherent in this approach with the use of an infinite sub-sequence of the Benjamin-Ono conservation laws to reach the $C^\infty-$ regularity.\\ In order to bring out a key preliminary result, we  give the following stochastic set up:\\
Consider the diffusion problem (also called the stochastic Benjamin-Ono-Burgers (BOB) equation)
\begin{equation}\label{BOBf}
\dt u+H\dy u+u\dx u=\alpha\dy u+\sqrt{\alpha}\eta, \ \ t>0,\ \ x\in \T,
\end{equation}
where $\eta$ is a stochastic force and $\alpha\in (0,1)$ is a viscosity parameter. 
In fact the problem $(\ref{BO})$ is the limit as $\alpha\to 0$ of $(\ref{BOBf}).$ A probabilistic global wellposedness for $(\ref{BOBf})$ is proved in Section $\ref{Pgwp}.$ Moreover, in Section $\ref{section_inv_visc}$, we establish  the existence of stationary solutions\footnote{Solutions to $(\ref{BOBf})$ whose laws are invariant along the time.} for this equation.  We present now the following smoothing property for stationary solutions:
\begin{nem}\label{nem_central}
Suppose that the noise $\eta$ is sufficiently regular in space. Let $u_\alpha$ be a stationary solution to $(\ref{BOBf})$ such that 
\begin{equation}
\E\|u_\alpha(t)\|^p<\infty \ \ \ \ \forall p\geq 2.\label{cond_p}
\end{equation}
Then 
\begin{equation}
\E\|u_\alpha(t)\|_n^2<\infty \ \ \ \ \forall n\geq 1.\label{resul_n}
\end{equation}
Moreover, if $(\ref{cond_p})$ holds uniformly in $\alpha$ then so does $(\ref{resul_n}).$
\end{nem}
The proof of this lemma relies on a combination of deterministic and probabilistic estimations based on the conservation laws of $(\ref{BO}).$

We prove in Section $\ref{section_inv_visc}$ that any stationary solution to $(\ref{BOBf})$ satisfies $(\ref{cond_p})$ uniformly in $\alpha$. Then, from $(\ref{resul_n})$ we conclude that stationary solutions to $(\ref{BOBf})$ are concentrated on $C^\infty$.
Passing to the limit as the viscosity goes to $0$, we find the main result of this paper (Theorems $\ref{thm_inv_mes}$, $\ref{non_den_plus}$, $\ref{Hausd}$ and $\ref{thm_gauss_dec}$):

\begin{thm}
There is a probability measure $\mu$ invariant under the flow of the BO equation $(\ref{BO})$ defined on $H^3(\T)$ and such that
\begin{equation*}
\mu(C^\infty(\T))=1.
\end{equation*}
Moreover, $\mu$ satisfies the following properties:
\begin{enumerate}
\item For any integer $n$, we have
\begin{equation*}
0<\int_{H^3} \|u\|_n^2\mu(du)<\infty.
\end{equation*}
\item There are constants $\sigma,C>0$ such that for any $R>0$ 
\begin{equation*}
\mu(u\in H^3,\ \|u\|\geq R)\leq Ce^{-\sigma R^2}.
\end{equation*}
\item There is an infinite sequence of conservation laws of the form $(\ref{form_energies})$ whose laws under $\mu$ are absolutely continuous with respect to the Lebesgue measure on $\R$.
\item The measure $\mu$ is  of at least $2$-dimensional nature in the sense that any compact set of Hausdorff dimension smaller than $2$ has $\mu$-measure $0$.
\end{enumerate}
\end{thm}
In fact, we expect infinite-dimensionality of the measure constructed here as in \cite{kuksin_nondegeul,KS12} concerning the $2-$dimensional Euler equations. To show this property in the context of the Benjamin-Ono equation, we have to prove some algebraic independence of the gradients of the conservation laws. In the present work, we face a technical difficulty in establishing such an independence for an arbitrary number of conservation laws. We propose a proof inspired by  \cite{kuksin_nondegeul,KS12} which works for the (at least) $2$-dimensionality. Then the infinite-dimensionality of $\mu$ remains an open question.

We deduce the following result by applying the Poincaré recurrence theorem.
\begin{cor}
For $\mu$-almost all $w$ in $C^\infty(\T)$, there is a sequence $\{t_k\}$ increasing to infinity such that 
\begin{equation*}
\lim_{k\to\infty} \|S_{t_k}w-w\|_n=0 \ \ \ \ for\ any\ n\geq 0.
\end{equation*}
Here $S_t$ denotes the flow of the Benjamin-Ono equation $(\ref{BO})$ on $H^3(\T)$.
\end{cor}

In the construction of such a measure, we use the control of Sobolev norms provided by the infinite sequence of conservation laws. The KdV and cubic $1-$dimensional NLS equations have infinitely many conservation laws whose structure is similar to $(\ref{form_energies})$ and our approach applies to these equations. Notice that  an infinite sequence of invariant Gaussian type measures of increasing regularity was constructed for KdV and cubic $1-$dimensional NLS equations in \cite{zhidkov2001korteweg, zhid_shrod}, we give then a kind of extension of this work to the $C^\infty(\T)$ space. However, the Benjamin-Ono equation is more difficult than these equations because of the weakness of its dispersion compared to KdV and the presence of a derivative in its non-linearity compared to NLS.
Then, here, we confine ourselves to the study of the BO equation which is less understood.

Let us briefly discuss an equation having infinitely many conservation laws but not admissible to the approach developed here. Consider the non-viscous Burgers equation
\begin{equation}\label{Burgers}
\dt u+u\dx u=0.
\end{equation}
It is easy to check that an infinite sequence of conservation laws is given by the quantities
\begin{equation*}
L_p(u)=\int u^p, \ \ \ p\geq 1.
\end{equation*}
Our approach does not apply to $(\ref{Burgers})$. This is due to its lack of dispersion which breaks the control of Sobolev's norms. 

\subsection*{Notation}
Let $A$ and $B$ be two positive quantities, we write
\begin{equation*}
A\lleq B
\end{equation*}
if there is a universal constant $\lambda\geq 0$ such that $A\leq \lambda B.$\\
For a real number $r$, we denote by $r^{+}$ (resp. $r^-$) the quantity $r+\epsilon$ (resp. $r-\epsilon$), where $\epsilon$ is a positive number close enough to $0$, while $r_+:=\max(r,0)$.\\
$\Z$ denotes the set of nonzero integers.\\
$\dot{H}(\T)=\{u\in L^2(\T):\ \ \int_\T u(x)dx=0\}.$\\
$\dot{H}^s(\T)=\{u\in\dot{H}(\T):\ \ D^su \in \dot{H}(\T)\},$ and $D^s$ is the $s$th derivative of $u$, where $s\geq 0$.\\
The $\dot{H}^s$-norm is denoted by  $\|.\|_s$ when $s>0$ and the $L^2$-norm is denoted by $\|.\|$.\\
For a functional $A(u)$, we denote the first and second derivatives of $A$ by $A'(u,v):=\partial_u A(u,v)=\partial A|_u(v)$ and $A''(u,v):=\partial_u^2 A(u,v)=\partial^2A|_{u}(v,v)$.\\
The sequence $\{e_n,\ n\in\Z\}$ is given by
$$e_n(x)=\left\{\begin{array}{l r c}\frac{\sin(nx)}{\sqrt{\pi}},\ \  \text{for}\ \ \ \ n>0,\\
\frac{\cos(nx)}{\sqrt{\pi}},\ \ \text{for}\ \ \ \ n<0.
\end{array}\right.$$
and forms an orthonormal basis of $\dot{H}(\T).$\\
$(\Omega,\mathcal{F},\P)$ is a complete probability space and $\mathcal{F}_t$ is a right-continuous filtration augmented with respect to $(\mathcal{F},\P).$
Given a sequence of real numbers $\{\lambda_n\}$ and a sequence of independent real standard Brownian motions $\{\beta_n(t)\}$ adapted to $\mathcal{F}_t,$ we set  
\begin{align}
\zeta (t,x) &=\sum_{n\in\Z}\lambda_n\beta_n(t)e_n(x),\label{brownm}\\
\eta(t,x) &=\frac{d}{dt}\zeta(t,x),\\
A_s &=\sum_{n\in\Z}\lambda_n^{2}n^{2s}.\label{Notat:A_s}
\end{align}
 
 \subsection*{Some stochastic results}\label{intro_section_convol_stoch_et_Ito_lemma}
The theorem and lemma below are useful ingredients in our work, we refer to \cite{karatshre} for their proofs.
\begin{thm}[Doob's optional theorem]\label{intro_Doob}
Let $x_t$ be a continuous $\mathcal{F}_t$-martingale  and $\tau\leq \sigma$ be two  $\mathcal{F}_t$-stopping times which are almost surely finite. Then
\begin{equation}
\E x_\tau=\E x_{\sigma}=\E x_0.
\end{equation} 
\end{thm}

\begin{nem}\label{intro_lemme_prog_mes}
Let $x_t$ be a continuous random process which is adapted to $\mathcal{F}_t$. Then $x_t(\omega)$ is adapted to $\mathcal{F}_t$.
\end{nem}

\paragraph{Stochastic convolution.}
Let $B$ be an operator on a separable Hilbert space $H$ with which we endow a Hilbert basis $\{e_m\}_{m\in\Z}$. Suppose that $\{e_m\}$ are eigenvectors of $B$ whose associated eigenvalues are $\{b_m\}\subset\Bbb C$, and moreover $|b_m|\to \infty$ as $m\to\infty.$   Suppose that
\begin{equation}
V_t(B):=\sum_{m\in\Z}\lambda_m^2\frac{|1-e^{2tb_m}|}{2|b_m|}<\infty\ \ \ for\ all\ t\geq 0,
\end{equation}
then the
quantity (which is called stochastic convolution)
\begin{equation}
\Theta_t(B):=\int_0^te^{(t-s)B}d\zeta(s,x):=\sum_{m\in\Z}\lambda_m\left(\int_0^te^{(t-s)b_m}d\beta_m(s)\right)e_m(x)\ \ \ t\geq 0
\end{equation}
is well defined in $H$. In fact $\Theta_t(B)$ is a continuous Gaussian process in $H$: for all $t\geq 0,$  $\Theta_t(B)\sim \mathcal{N}_H(0,V_t(B))$.  

\begin{rmq}\label{intro_rmq_reg_stoch_convol}
If $\text{Re}(b_m)<0,$ then the sequence $\{|1-e^{2tb_m}|/|2b_m|\}$ is bounded (even uniformly in $t$), therefore $\Theta_t(B)$ is well defined in  $H$ as soon as $\sum_{m}\lambda_m^2<\infty.$

 The concrete case that is studied in this paper is the Hilbert space $L^2(\T)$ and an operator of type $-\dy$ (more exactly $B=-(H-\alpha)\dy$). With the use of the Itô isometry, we have
\begin{equation}
\E\|\Theta_t(B)\|_s^2=\sum_{m\in\Z}m^{2s}\lambda_m^2\frac{|1-e^{2tb_m}|}{2|b_m|}, \ \ \ b_m=-(\i\osgn(m)+\alpha)m^2.
\end{equation} 
Then $\text{Re}(b_m)<0$ for any $m\in\Z$. Therefore $\Theta_t(B)\in H^s$ almost surely as soon as
\begin{equation}
\sum_{m\in\Z}m^{2s}\lambda_m^2<\infty.
\end{equation}
\end{rmq}
In that case, as a Gaussian random variable in $H^s$, the stochastic convolution $\Theta_t$ verifies the Fernique theorem; that is, there is a constant $c_s$ such that
\begin{equation}
\E e^{c_s\|\Theta_t\|_{H^s}^2}<\infty \ \ \ \ for \ \ all \ \ t\geq 0.
\end{equation}
\paragraph{Stochastic well-posedness and the Itô property relative to a Gelfand triple.}
Let us consider the following stochastic PDE:
\begin{equation}\label{intro_EDPS_gen}
du_t=(Lu+f(u))dt+d\zeta,
\end{equation}
where $L$ is a differential operator, $f$ is a function possibly nonlinear in $u$ and $\zeta$ is a Brownian motion defined as in $(\ref{brownm}).$

\begin{defi}\label{intro_SGWP}
Let $s\in\R$. Equation $(\ref{intro_EDPS_gen})$ is said to be stochastically (globally) well-posed in $H^s$ if for all $T>0$ the following properties hold
\begin{enumerate}
\item  for any random variable $u_0$ in $H^s$ which is independent of $\mathcal{F}_t,$ we have, for almost all $\omega\in\Omega$,
\begin{enumerate}
\item (Existence) there exists $u:=u^\omega \in \Lambda_T(s):=C(0,T;H^s)\cap L^2(0,T;H^{s+1})$ satisfying the relation 
\begin{equation}
u(t)=u_0+\int_0^t(Lu_s+f(u_s))ds+\zeta(t)\ \ \text{for all $t\in[0,T]$}\label{intro_sol_mild}
\end{equation}
in $H^{s-1}.$
We denote this solution by $u(t,u_0):=u^\omega(t,u_0).$
\item (Uniqueness) if $u_1,u_2\in \Lambda_T(s)$ are two solutions in the sense of $(\ref{intro_sol_mild}),$ then $u_1\equiv u_2$ on $[0,T],$ 
\end{enumerate}
\item (Continuity w.r.t. initial data) for almost all $\omega,$ we have
\begin{equation}
\lim_{u_{0}\to u_{0}'}u(.,u_0)=u(.,u_0') \ \ \ \text{in $\Lambda_T(s)$},
\end{equation}
where $u_0$ and $u_0'$ are deterministic data in $H^s$;
\item the process $(\omega,t)\mapsto u^\omega(t)$ is adapted to the filtration $\sigma(u_0,\mathcal{F}_t)$.

\end{enumerate}
\end{defi}
\begin{rmq}\label{intro_rmq_def_SWP}
In what follows we call $(H^{s-1},H^s,H^{s+1})$ a Gelfand triple. The process $u_t$ described in Definition $\ref{intro_SGWP}$ satisfies the following properties:
\begin{itemize}
\item Considered as a process in $H^s$, it is progressively measurable w.r.t. $\sigma(u_0,\mathcal{F}_t);$\\
this follows from continuity of $u_t$ and the Lemma $\ref{intro_lemme_prog_mes}.$
\item It satisfies the Feller property, being continuous in $t$ and w.r.t. initial data.
\item It is a Markov process: Set $$P_t(w,\Gamma):=\P(u(t,w)\in\Gamma|u(0)=w);$$ then $P_t$ satisfies the so called Chapman-Kolmogorov relation. Let us write down the corresponding Markov semi-groups.
\begin{equation}
\mathfrak{P}_tf(v)=\int_{H^s}f(w)P_t(v,dw)  \ \ \ C_b(H^s)\to C_b(H^s),
\end{equation}
\begin{equation}
\mathfrak{P}_t^*\mu(\Gamma)=\int_{H^s}\mu(dw)P_t(w,\Gamma)\ \ \ \ \mathfrak{p}(H^s)\to\mathfrak{p}(H^s).
\end{equation}
Here, $C_b(H^s)$ is the space of bounded continuous functions on $H^s$, $\mathfrak{p}(H^s)$ is the set of probability measures on $H^s.$ These maps satisfy the duality relation
\begin{equation}
(\mathfrak{P}_tf,\mu)=(f,\mathfrak{P}_t^*\mu).
\end{equation}
\end{itemize}
\end{rmq}
Now, let us introduce the following definition:
\begin{defi}\label{intro_WS}
We say that $(\ref{intro_EDPS_gen})$ has the Itô property on the Gelfand triple $(H^{s-1},H^s,H^{s+1})$ if
\begin{enumerate}
\item it is stochastically well-posed on $H^s$;
\item  
the process $h:=Lu+f(u)$ is $\mathcal{F}_t$-adapted and
\begin{align}
\P\left(\int_0^t(\|u(r)\|_{s+1}^2+\|h(r)\|_{s-1}^2)dr< \infty, \ \ \forall\ t>0\right)=1,\ \ 
\sum_{m \in\Z}m^{2s}\lambda_m^2<\infty.\label{intro_hyp_KS}
\end{align}
\end{enumerate}
\end{defi}
\begin{rmq}
Our definition of the Itô property is different from what we find in some literature. But the interest of our choice is that part $(2)$ gathers "good" properties of a process allowing us to apply a version of the Itô formula proved in  Section $A.7$ (Theorem $A.7.5$ and Corollary $A.7.6$) of \cite{KS12}. Below, we present that formula.
\end{rmq}
\begin{thm}[see Section $A.7$ of \cite{KS12}]\label{intro_KS_Ito_theorem}
Let $F\in C^2(H^s,\R)$ be a functional which is uniformly continuous, together with its first two derivatives, on any ball of $H^s$. Suppose that $F$ satisfies the following conditions:
\begin{enumerate}
\item There is a function $K:\R_+\to\R_+$ such that
\begin{equation}\label{intro_KS_Ito_F1}
|\nabla_uF(u;v)|\leq K(\|u\|_{s})\|u\|_{{s+1}}\|v\|_{{s-1}},\ \ \ u\in H^{s+1},\ \ v\in H^{s-1}.
\end{equation}
\item For any sequence $\{w_k\}\subset H^{s+1}$ converging toward $w\in H^{s+1}$ and any $v\in H^{s-1}$, we have
\begin{equation}\label{intro_KS_Ito_F2}
\nabla_u F(w_k;v)\to\nabla_uF(w;v),\ \ as\ \ k\to\infty.
\end{equation}
\item
\begin{equation}\label{intro_KS_Ito_F3}
\sum_{m\in\Z}a_m^2\E\int_0^t|\nabla_u F(u;e_m)|^2ds <\infty\ \ \ \ for\ all \ t>0.
\end{equation}
\end{enumerate}
Then we have
\begin{align}\label{intro_KS_Ito}
F(u(t))=F(u(0))+\int_0^t\left(\du F(u(s);f(s))+\frac{1}{2}\sum_{m\in\Z}\partial_u^2F(u(s),g_m)\right)ds \nonumber\\
+\sum_{m\in\Z}\int_0^t\du F(u(s),g_m)d\beta_m(s).
\end{align}
In particular,
\begin{equation}
\E F(u(t))=\E F(u(0))+\int_0^t\E\left(\du F(u(s),f(s))+\frac{1}{2}\sum_{m\in\Z}\partial_u^2F(u(s),g_m)\right)ds.
\end{equation}
If one omits $(\ref{intro_KS_Ito_F3}),$ then we have th3e formula $(\ref{intro_KS_Ito})$ in which $t$ is replaced by the stopping time $t\wedge \tau_n$ where
\begin{equation}
\tau_n=\inf\{t\geq 0,\ \|u(t)\|_{s}> n\}, \ \ n\in\N,
\end{equation}
with the convention $\inf\emptyset=+\infty.$
\end{thm}

\section{Deterministic estimates}
\subsection*{Conservation laws}
Following \cite{tzvetkov-visc}, we define the following subsets of $C^\infty(\T)$:
\begin{align*}
\mathcal{P}_1 &=\{\partial_x^\alpha u,\ \partial_x^\alpha Hu\ | \ \alpha\in\N\},\\
\mathcal{P}_2 &=\{(\partial_x^{\alpha_1}Z_1u)(\partial_x^{\alpha_2}Z_2u)\ |\ \alpha_i\in\N,\ Z_i\in\{Id,H\}\}.
\end{align*}
Let us define in a generic manner the sets $\mathcal{P}_n \ n\geq 3$ containing the functions of the form
\begin{equation}
p_n(u)=\prod_{i=1}^k Z_i(p_{j_i}(u)),\ \text{where}\ Z_i\in\{Id,H\},\ \ \sum_1^kj_i=n,\ \ p_{j_i}\in\mathcal{P}_{j_i}, \ 2\leq k\leq n,\ \ j_i<n.\label{form_pn}
\end{equation} 
To a function $p_n(u)$ of the form $(\ref{form_pn}),$ we associate the function
\begin{equation}
\tilde{p}_n(u)=\prod_{i=1}^kp_{j_i}(u),\label{form_pn-tild}
\end{equation}
and we set the quantities
\begin{align*}
S(p(u)) &=\sum_{i=1}^n\alpha_i,\\
M(p(u)) &=\max_{1\leq i\leq n}\alpha_i.
\end{align*}

The following is a description given in \cite{tzvetkov-visc} for the integer order remainder terms:
\begin{equation}\label{reste_loi}
R_n(u)=\sum_{\substack{p(u)\in \mathcal{P}_3 \\ \tilde{p}(u)=u\partial_x^{n-1}u\partial_x^nu}} c_{n}(p)\int p(u)
 +\sum_{\substack{p(u)\in \mathcal{P}_j\ j=3,...,2n+2 \\ S(p(u))=2n-j+2\\ M(p(u))\leq n-1}}c_{n}(p)\int p(u),
\end{equation}
where $c_n(p)$ are some constants. The first three integer order conservation laws are
\begin{align*}
E_0(u)&=\int u^2,\\
E_1(u)&=\int (\dx u)^2+\frac{3}{4}\int u^2H\dx u+\frac{1}{8}\int u^4,\\
E_2(u) &= \int (\dy u)^2-\frac{5}{4}\int \left((\dx u)^2H\dx u+2\dy uH\dx u\right)\\
&+\frac{5}{16}\int\left(5u^2(\dx u)^2+u^2(H\dx u)^2+2uH(\dx u)H(u\dx u)\right)\\
&+\int\left(\frac{5}{32}u^4H(\dx u)+\frac{5}{24}u^3H(u\dx u)\right)+\frac{1}{48}\int u^6.
\end{align*}

\subsection*{Estimates}
 Let us give some properties for the integer order conservation laws of the Benjamin-Ono equation.
\begin{nem}
For any integer $n\geq 1,$ there are $c_n^-,c_n^+>0$ such that for all $u$ in $H^n(\T)$ 
\begin{equation}
\frac{1}{2}\|u\|_n^2-c_n^-\|u\|^{2n+2}\leq E_n(u)\leq 2\|u\|_n^2+c_n^+\|u\|^{2n+2}.\label{apriori_est_on_En}
\end{equation}
\end{nem}

\begin{nem}\label{lemme_central}
For all $\epsilon>0,$ there is $C_\epsilon>0$ such that for all $u$ in $H^{n+1}(\T)$
\begin{equation*}
E_n'(u,\dy u)\leq (-2+\epsilon)\|u\|_{n+1}^2+C_\epsilon\|u\|(1+\|u\|)^{b_n},
\end{equation*}
where $b_n$ depends only on $n.$
\end{nem}

\begin{rmq}\label{remark_nouvelles_lois}
 Since the $L^2$-norm is preserved by $(\ref{BO})$ we can deduce from $(\ref{apriori_est_on_En})$ and the arguments of the proof of Lemma $\ref{lemme_central}$, by adding appropriate polynomials of $\|u\|$, new conservation laws $E_n^*(u)$ and $\tilde{E}_n(u)$ satisfying
\begin{align*}
0\leq \|u\|_n^2 &\leq E_n^*(u),\\
0\leq \|u\|_n^2 &\leq \tilde{E}_n'(u,u).
\end{align*} 
\end{rmq}

Inequalities $(\ref{apriori_est_on_En})$ can be established using arguments similar to those of the proof of Lemma $\ref{lemme_central}$.

\begin{proof}[Proof of Lemma $\ref{lemme_central}$]
Taking into account of the properties of the Hilbert transform such as continuity on $H^s$ and $L^p$ ($s\geq 0$, $p\in]1,\infty[$), we can neglect its effect for our purpose and just consider the functions 
\begin{align*}
R_n^1(u) &=\int u\partial_x^{n-1}u\partial_x^nu,\\
R_n^{2,j}(u) &=\int \prod_{i=1}^j\partial_x^{\alpha_i}u,\ \ \ j=3,...,2n+2, \ \ \sum_{i=1}^j\alpha_i=2n+2-j.
\end{align*}
Here $R_n^1(u)$ corresponds to the first term of $(\ref{reste_loi})$ and the second term of $(\ref{reste_loi})$ can be estimated considering the quantities $R_n^{2,j}(u)$.
Set 
\begin{equation*}
R_n^0=\|u\|_n^2.
\end{equation*}

\paragraph{Estimates concerning $R_n^0$:}
\begin{equation}
\partial_u R_n^{0}(u,\dy u)=-2\|u\|_{n+1}^2.\label{estim_R_n^0}
\end{equation}
\paragraph{Estimates concerning $R_n^1$:}
\begin{align*}
\partial_u R_n^{1}(u,\dy u) &=\int \partial_x^2u\partial_x^{n-1}u\partial_x^{n}u+\int u\partial_x^{n+1}u\partial_x^nu+\int u\partial_x^{n-1}u\partial_x^{n+2}u\\
&=\int \partial_x^2u\partial_x^{n-1}u\partial_x^{n}u-\int\dx u\partial_x^{n-1}u\partial_x^{n+1}u\\
&=2\int \partial_x^2u\partial_x^{n-1}u\partial_x^{n}u+\int\dx u(\partial_x^{n}u)^2
=I+II.
\end{align*}
Let $\gamma_i,\ i=1,2,3$ be three positive numbers satisfying $\sum_{i=1}^3\frac{1}{\gamma_i}=1,$ we apply the generalised Hölder formula with them to find
\begin{align*}
|I|\leq\|\partial_x^2u\|_{L^{\gamma_1}}\|\partial_x^{n-1}u\|_{L^{\gamma_2}}\|\partial_x^{n}u\|_{L^{\gamma_3}}.
\end{align*}
By the embedding inequality $\|.\|_{L^{\gamma_i}}\lleq \|.\|_{1/2-1/\gamma_i}$, we get
\begin{equation*}
|I|\lleq\|u\|_{5/2-1/\gamma_i}\|u\|_{-1/2-1/\gamma_i+n}\|u\|_{1/2-1/\gamma_i+n}.
\end{equation*}
Now interpolate between $L^2$ and $H^{n+1}$ to find
\begin{align*}
|I|\leq C_1\|u\|^{d_1}_{n+1}\|u\|^{3-d_1}
\end{align*}
where
\begin{align*}
d_1=\frac{2n+3}{2(n+1)}<2.\ \
\end{align*}
One can establish the same control (with same $d_1$) for $|II|$ by remarking that
\begin{equation*}
|II|\lleq \|u\|_{1}\|\partial_x^{n}u\|_{L^4}^2\lleq \|u\|_{1}\|u\|_{n+1/4}^2\lleq \|u\|_{n+1}^{\frac{(n+1-1)+2(n+1-n-1/4)}{n+1}}\|u\|^{c}.
\end{equation*}
Then for suitable $b_1$
\begin{equation}
|\du R_n^1(u,\dy u)|\leq \epsilon\|u\|_{n+1}^2+C^1_\epsilon\|u\|^{b_1}.\label{estim_R_n^1}
\end{equation}
\paragraph{Estimates concerning $R_n^{2,j}$:}
\begin{equation*}
\du R_n^{2,j}(u,\dy u)=\int \prod_{i=1}^j\partial_x^{\alpha_i}u\ \ \ \ j=3,...,2n+2,
\end{equation*}
where $\sum_{i=1}^j\alpha_i=2n-j+4$ and $\max_{1\leq i\leq j}\alpha_i\leq n+1.$

We follow two complementary cases:
\begin{itemize}
\item Case $1$: $\max_{1\leq i\leq j}\alpha_i\leq n$. Let $(\gamma_i)$ be  $j$ real numbers such that $\sum_{i=1}^j \frac{1}{\gamma_i}=1$. Then the generalised Hölder formula combined with usual interpolation inequalities shows:
\begin{equation*}
|\du R_n^{2,j}(u,\dy u)|\leq C\prod_{i=1}^j\|u\|_{\kappa_i},
\end{equation*}
where $\kappa_i=\frac{1}{2}-\frac{1}{\gamma_i}+\alpha_i.$
Then
\begin{align*}
|\du R_n^{2,j}(u,\dy u)|\leq C\prod_{i=1}^j\|u\|^{\frac{n+1-\kappa_i}{n+1}}\|u\|^{\frac{\kappa_i}{n+1}}_{n+1}.
\end{align*}
We remark now that 
$$\sum_{i=1}^j\kappa_i=\sum_1^j\left(\frac{1}{2}-\frac{1}{\gamma_i}+\alpha_i\right)=2n+3-\frac{j}{2}.$$ Then 
$$\sum_{i=1}^j\frac{\kappa_i}{n+1}=\frac{2n+3-\frac{j}{2}}{n+1}<2.$$ Thus for suitable $b_2,$
\begin{equation*}
|\du R_n^{2,j}(u,\dy u)|\leq \epsilon\|u\|_{n+1}^2+C^2_\epsilon\|u\|^{b_2}.
\end{equation*}
\item Case $2:$ $\alpha_1=n+1$. Then $\sum_{i=2}^j\alpha_i=n-j+3\leq n,$ and we have
\begin{equation*}
|\du R_n^{2,j}(u,\dy u)|\leq \|u\|_{n+1}\left(\int\prod_{i=2}^j|\partial_x^{\alpha_i}u|^2\right)^\frac{1}{2}.
\end{equation*}
Take again $(\gamma_i)$ such that $\sum_{i=2}^j\frac{1}{\gamma_i}=1$. Then
\begin{align*}
|\du R_n^{2,j}(u,\dy u)|&\leq \|u\|_{n+1}\prod_{i=2}^j\|\partial_x^{\alpha_i}u\|_{L^{2\gamma_i}}\\
&\leq \|u\|_{n+1}\prod_{i=2}^j\|u\|_{\kappa_i}, \ \ \ \kappa_i=\frac{1}{2}-\frac{1}{2\gamma_i}+\alpha_i,\\
&\leq\|u\|_{n+1}\prod_{i=2}^j\|u\|^{\frac{n+1-\kappa_i}{n+1}}\|u\|_{n+1}^{\frac{\kappa_i}{n+1}}.
\end{align*}
Since $\sum_{i=2}^j\kappa_i= n+2-\frac{j}{2}\leq n+\frac{1}{2}$, we have $\frac{1}{n+1}\sum_{i=2}^j\kappa_i<1$ and the existence of a suitable $b_3$ such that
\begin{equation}
|\du R_n^{2,j}(u,\dy u)|\leq \epsilon\|u\|_{n+1}^2+C^3_\epsilon\|u\|^{b_3}.\label{estim_r_n_2,j}
\end{equation}
\end{itemize}
Combining $(\ref{estim_R_n^0}),(\ref{estim_R_n^1})$ and $(\ref{estim_r_n_2,j})$ with a good choice of $\epsilon$, we have the claim.
\end{proof}

\section{IVP of the stochastic BOB equation}\label{Pgwp}
Consider the initial value problem concerning the stochastic BOB equation $(\ref{BOBf})$

\begin{align}\label{equ}
\left\{\begin{array}{l r c}
\dt u+H\dy u+u\dx u=\alpha\dy u+\sqrt{\alpha}\eta \ \ \ t>0,  \\
 u|_{t=0}=u_0.
\end{array}
\right.
\end{align}
Recall that, for $s\geq 0$,
\begin{equation*}
A_s=\sum_{m\in\Z}m^{2s}\lambda_m^2,
\end{equation*}
these quantities measure the regularity in space of the noise. Namely,
$$A_s<+\infty \ \Leftrightarrow \ \eta(t,.)\in \dot{H}^s.$$

\subsection*{Stochastic wellposedness, well-structuredness}
\begin{prop}\label{wellposedness}
Let $s\geq 2$ be an integer. Suppose $A_{s}$ is finite. Then the problem  $(\ref{equ})$ is stochastically globally wellposed in $\dot{H}^s(\T)$  in the sense of Definition $\ref{intro_SGWP}.$
\end{prop}

In order to prove the existence result in Proposition $\ref{wellposedness}$, we split the problem $(\ref{equ})$ as follows:

\begin{itemize}
\item A linear stochastic problem:
\begin{align}\label{equ_lin_stoc}
\left\{\begin{array}{l r c}
\dt z_\alpha+H\dy z_\alpha=\alpha\dy z_\alpha+\sqrt{\alpha}\eta \ \ \ t>0,  \\
 z_{\alpha}|_{t=0}=0.
\end{array}
\right.
\end{align}
\item A nonlinear deterministic problem:
\begin{align}\label{equ_nonlin}
\left\{\begin{array}{l r c}
\dt v+H\dy v+(v+z_\alpha)\dx (v+z_\alpha)=\alpha\dy v \ \ \ t>0,  \\
 v|_{t=0}=u_0.
\end{array}
\right.
\end{align}
Here $z_\alpha$ is a realization of a solution of $(\ref{equ_lin_stoc}).$
\end{itemize}

For $z_\alpha$ and $v$ respective solutions of $(\ref{equ_lin_stoc})$ and $(\ref{equ_nonlin})$, it is easy to see that $u=v+z_\alpha$ is a solution of $(\ref{equ}).$ The linear problem $(\ref{equ_lin_stoc})$ is solved by the stochastic convolution (see section $\ref{intro_section_convol_stoch_et_Ito_lemma}$):
\begin{equation}\label{sol_lin_pbm_z}
z_\alpha(t)=\sqrt{\alpha}\int_0^te^{-(t-s)(H-\alpha)\dy}d\zeta(s)=:\sqrt{\alpha}z(t).
\end{equation}
Remark that, as defined, the function $z$ still depends on $\alpha$. But all its Sobolev norms are uniformly controlled with respect to $\alpha$, this justifies that abuse of notation.\\
If for some $s\geq 0$  $A_s$ is finite, then we have for all $T>0$
\begin{equation}\label{classe_lin_pbm}
z\in \Lambda_T(s):=C([0,T],\dot{H}^s(\T))\cap L^2([0,T],\dot{H}^{s+1}(\T)) \ \ \ \text{for $\Bbb P-a.e.\ \ \omega\in\Omega.$}
\end{equation}
Uniqueness of solution for the problem $(\ref{equ_lin_stoc})$ is obtained by standard arguments. Moreover
if we suppose $A_n$ finite, we can apply the Itô formula to the $\dot{H}^n-$norms (which are preserved by the linear Benjamin-Ono) to find that
\begin{equation}
\E\|z_\alpha\|_n^2+2\alpha\int_0^t\E\|z_\alpha\|_{n+1}^2ds=\alpha A_nt.\label{ito_sur_linear2}
\end{equation} 
Denoting by $z^m$ the projection $(z,e_m)$, we have that
\begin{align*}
z^m(t)=\lambda_m\int_0^te^{m^2(t-s)(\i\text{sgn}(m)-\alpha)}d\beta_m(s).
\end{align*}
Since the function $s\to e^{m^2(t-s)(\i\text{sgn}(m)-\alpha)}$ is $C^1$, we employ a usual (stochastic) integration by parts formula to obtain
\begin{align*}
z^m(t)=\lambda_m\beta_m(t)+m^2(\i\text{sgn}(m)-\alpha)\lambda_m\int_0^te^{m^2(t-s)(\i\text{sgn}(m)-\alpha)}\beta(s)sds.
\end{align*}
Then we arrive at
\begin{align*}
\sup_{t\in[0,T]}|z^m(t)|^2\leq 2\lambda_m^2[1+(1-\alpha)^2m^4T^2]\sup_{t\in[0,T]}|\beta_m(t)|^2\leq 2\lambda_m^2 [1+m^4T^2]\sup_{t\in[0,T]}|\beta_m(t)|^2
\end{align*}
After summing in $m$, we arrive at
\begin{equation*}
\sup_{t\in[0,T]}\|z(t)\|^2\lleq_T \sup_{t\in[0,T]}\|\zeta(t)\|_2^2.
\end{equation*}
More generally, for any $m$ such that $A_{m+2}$ is finite, we have
\begin{equation*}
\sup_{t\in[0,T]}\|z(t)\|^2_m\lleq_T \sup_{t\in[0,T]}\|\zeta(t)\|_{m+2}^2,
\end{equation*}
and finally
\begin{equation}\label{estim_sup_z}
\sup_{t\in[0,T]}\|z_\alpha(t)\|^2_m\lleq_T \alpha\sup_{t\in[0,T]}\|\zeta(t)\|_{m+2}^2.
\end{equation}
\begin{prop}\label{existence_equ_nonlin}
Let $s\geq 2$ be an integer, and suppose $A_s<\infty$. Let $u_0$ be a random variable in $\dot{H}^s(\T)$ independent of $\mathcal{F}_t$. Then for any $T>0$, for $a.e$ $\omega$, the nonlinear problem $(\ref{equ_nonlin})$ associated to $u_0$ admits a solution in $\Lambda_T(s).$ Moreover the process solution is adapted to $\sigma(u_0,\mathcal{F}_t)$.
\end{prop}
Proposition $\ref{existence_equ_nonlin}$ is proved combining the two paragraphs below:

\paragraph{A priori estimates.}
The following lemma is proved using the first three integer order (modified) conservation laws $E_n^*(u)$ of the Remark $\ref{remark_nouvelles_lois}$, its proof is presented in the appendix.
\begin{nem}\label{estim_123}
For any $T>0$, for almost any realization of $z$ we have the following a priori estimates for the nonlinear problem $(\ref{equ_nonlin})$
\begin{equation}\label{est_norm_inf2}
\sup_{t\in[0,T]}\|v(t)\|_i^2+\alpha\int_0^T\|v(t)\|_{i+1}^2dt\leq C\left(T,\|u_0\|_i,\|z\|_{L^\infty(0,T;H^i)}\right)\ \ \ \ i=0,1,2,
\end{equation}
where $C$ does not depend on $\alpha\in (0,1).$
\end{nem}
Since $H^2(\T)$ is continuously embedded in $C^1(\T)$, we infer
\begin{cor}
For any $T>0$, for almost any realization of $z$, and for any initial datum $u_0\in H^2$, a solution $v$ to $(\ref{equ_nonlin})$ satisfies
\begin{equation}
\sup_{t\in[0,T]}\|\dx v(t)\|_{L^\infty}\leq C\left(T,\|u_0\|_2,\|z\|_{L^\infty(0,T;H^2)}\right),\label{A}
\end{equation}
where $C$ does not depend on $\alpha\in (0,1).$
\end{cor}

\begin{nem}
For any $T>0$, for any integer $s>2$, and for almost any realization of $z$ we have the higher order a priori estimates for $(\ref{equ_nonlin})$
\begin{equation}\label{estim_s>2}
\sup_{t\in[0,T]}\|v(t)\|_s^2+\alpha\int_0^T\|v(t)\|_{s+1}^2dt\leq C\left(T,\|u_0\|_s,\|z\|_{L^\infty(0,T;H^s)}\right),
\end{equation}
where $C$ does not depend on $\alpha\in (0,1).$
\end{nem}
Before giving the proof of the estimate $(\ref{estim_s>2})$, let us prove the following commutator estimate:
\begin{nem}
Let $s\geq 3$ be an integer and $v$ be in $H^{s+1}.$ We have
\begin{equation}
\|[\partial_x^s,v]\dx v\|\lleq\|v\|_2\|v\|_{s},\label{notre_kp}
\end{equation}
where $[\partial_x^s,v]\dx v=\partial_x^s(v\dx v)-v\partial_x^s(\dx v).$
\end{nem}
\begin{proof}
By the Leibniz rule we have
\begin{align*}
[\partial_x^s,v]\dx v=\sum_{k=1}^{s}\binom{s}{k}\partial_x^kv\partial_x^{s+1-k}v.
\end{align*}
We separate the above sum into three general terms:
\begin{enumerate}
\item We have $k\in\{1,\ s\}$ if and only if the general term is $\dx v\partial_x^sv$. By using the embedding $H^1\subset L^\infty$, we have the inequality
\begin{align*}
\|\dx v\partial_x^sv\|\leq\|v\|_2\|v\|_{s}.
\end{align*}
\item We have $k\in\{2,\ s-1\}$ if and only if the general term is $\dy v\partial_x^{s-1}v.$ We have (always by $H^1\subset L^\infty$)
\begin{align*}
\|\dy v\partial_x^{s-1}v\|\leq\|v\|_2\|v\|_{s}.
\end{align*}
\item When $s\geq 5$ we have the last situation which is $3\leq k\leq s-2$, we have then $3\leq s+1-k\leq s-2$ as well. We estimate the corresponding general term as follows
\begin{align*}
\|\partial_x^kv\partial_x^{s+1-k}\|\leq\|v\|_{k+1}\|v\|_{s+1-k}\lleq \|v\|_2^{\frac{s-k-1}{s-2}}\|v\|_s^{\frac{k-1}{s-2}}\|v\|_2^{\frac{k-1}{s-2}}\|v\|_s^{\frac{s-k-1}{s-2}}=\|v\|_2\|v\|_s.
\end{align*}
\end{enumerate}
We complete the proof after taking a weighted sum of these terms.
\end{proof}

\begin{proof}[Proof of the estimate $(\ref{estim_s>2})$]
We recall the non-linear equation satisfied by $v$:
\begin{equation*}
\dt v+H\dy v-\alpha\dy v=-v\dx v-\dx (vz_\alpha)-\frac{1}{2}\dx z_\alpha^2.
\end{equation*}
Then for an integer $s>2$, we have
\begin{align*}
(\partial_x^sv,\partial_x^s\dt v)+\alpha(\partial_x^{s+1}v,\partial_x^{s+1}v)= -(\partial_x^sv,\partial_x^s(v\dx v))\underbrace{-(\partial_x^sv,\partial_x^{s+1}(vz_\alpha))}_{=+(\partial_x^{s+1}v,\partial_x^{s}(vz_\alpha))}+\frac{1}{2}(\partial_x^{s+1}v,\partial_x^sz_\alpha^2).
\end{align*}
Therefore
\begin{align*}
\frac{1}{2}\dt\|v\|_s^2+\alpha\|v\|_{s+1}^2= I+II+III.
\end{align*}
Using the commutator estimate $(\ref{notre_kp})$ and the algebra structure of $H^s(\T)$, we have
\begin{align*}
|I| &=|(\partial_x^sv,\partial_x^s(v\dx v)-v\partial_x^s\dx v)+(\partial_x^sv,v\partial_x^s\dx v)|\\
&=|(\partial_x^sv,[\partial_x^s,v]\dx v)-\frac{1}{2}(\dx v,|\partial_x^sv|^2)|\\
&\lleq \|v\|_s^2\|v\|_2.
\end{align*}
By Cauchy-Schwarz and the algebra structure of $H^s$, we have
\begin{equation*}
|II|+|III|\leq \frac{\alpha}{2}\|v\|_{s+1}^2+C_1\|v\|_s^2\|z\|_s^2+C_2\alpha\|z\|_s^4,
\end{equation*}
where $C_1$ and $C_2$ depend only on $s$.
It remains to combine the Gronwall lemma with $(\ref{est_norm_inf2})$  to get the claim. 
\end{proof}
\paragraph{Local and global existence for the nonlinear problem $(\ref{equ_nonlin})$.}

Let $s\geq 2$. For a positive $T$ the space $\Lambda_T(s)$ is endowed with the norm defined by
\begin{equation}
\|u\|_{\Lambda_T(s)}=\sup_{t\in[0,T]}\left(e^{-\frac{t}{T}}\left\{\|u(t)\|_s^2+\alpha\int_0^t\|u(r)\|_{s+1}^2dr\right\}\right)^{\frac{1}{2}}.\label{norm_def}
\end{equation}
Let $R>0$, denote by $B_R$ the ball in $H^s$ of center $0$ and radius $R$.
\begin{rmq}
The factor $e^{-\frac{t}{T}}$ in $(\ref{norm_def})$ is introduced just for convenience in the computations. The norm defined in $(\ref{norm_def})$ is actually equivalent to the one without that factor.
\end{rmq}

\begin{prop}\label{local_exist_nonlinear}
Let $s\geq 2$ and $\alpha\in (0,1)$. For all $R>0$, there is $T_R>0$ such that for any $u_0$ in $B_{R/2}$, the nonlinear problem $(\ref{equ_nonlin})$ has a unique solution in $\Lambda_{T_R}(s).$
\end{prop}

\begin{rmq}
We combine the local existence of Proposition $\ref{local_exist_nonlinear}$, Lemma $\ref{estim_123}$, and estimate $(\ref{estim_s>2})$ to get the global existence for $(\ref{equ_nonlin}).$
\end{rmq}

\begin{proof}[Proof of Proposition $\ref{local_exist_nonlinear}$]
Let us look for a fixed point of the map
$$\mathfrak{F}v=e^{-t(H-\alpha)\dx^2}u_0-\int_0^te^{-(t-s)(H-\alpha)\dx^2}(z_\alpha+v)\dx(z_\alpha+v)ds.$$
 We proceed as follows:

\begin{itemize}
\item Step 1: We prove that for any $R>0$, there is $T>0$ such that the ball $B_{T,s}$ of $\Lambda_T(s)$ centered at $0$ and of radius $R$ satisfies $\f(B_{T,s})\subset B_{T,s}$ if $\|u_0\|_s\leq R/2$.

\begin{align*}
-\frac{1}{2}\frac{d}{dt}\|\mathfrak{F}v\|_s^2 &=-(\dt D^s\mathfrak{F}v,D^s\f(v))\\
&=-((H-\alpha)D^{s+1}\f(v),D^{s+1}\f(v))+\frac{1}{2}(D^s(z_\alpha+v)^2,D^{s+1}\f(v))\\
&\geq \alpha\|\f(v)\|_{s+1}^2-\frac{1}{2}\|z_\alpha+v\|_{s}^2\|\f(v)\|_{s+1}\\
&\geq \alpha\|\f(v)\|_{s+1}^2-\frac{\alpha}{2}\|\f(v)\|_{s+1}^2-\frac{C}{\alpha}(\|z_\alpha\|_s^4+\|v\|_s^4).
\end{align*}
 Then there is an universal constant $c>0$ such that
 \begin{align*}
\frac{d}{dt}\|\f(v)\|_s^2+\alpha\|\f(v)\|_{s+1}^2 &\leq \frac{c}{\alpha}e^{\frac{2t}{T}}(R^4+\|z_\alpha\|_{\Lambda_T(s)}^4).
\end{align*}
Thus, after integration with respect to $t$, we find
\begin{align*}
\|\f(v)\|_s^2+\alpha\int_0^t\|\f(v)\|_{s+1}^2ds &\leq \|u_0\|_s^2 +\frac{\tilde{c}T}{\alpha}e^{\frac{t}{T}}(R^4+\|z_\alpha\|_{\Lambda_T(s)}^4).
\end{align*}
Multiplying the last relation by $e^{-\frac{t}{T}},$ it remains to choose $T$ small enough so that we obtain the claimed result. 

\item Step 2: We now prove that $\f$ is a contraction on the ball constructed above. We have
$$\dt\mathfrak{F}v=-\{(v+z_\alpha)\dx(v+z_\alpha)+(H-\alpha)\dx^2\mathfrak{F}v\},$$
then for $v_1$ and $v_2$ in $\Lambda_T(s)$, we have
\begin{align*}
-\frac{1}{2}\frac{d}{dt}\|\mathfrak{F}v_1-\mathfrak{F}v_2\|_s^2 &=-(\dt D^s(\mathfrak{F}v_1-\mathfrak{F}v_2),D^s(\mathfrak{F}v_1-\mathfrak{F}v_2))\\
&=(D^s(F_z(v_1)-F_z(v_2)),D^{s+1}(\mathfrak{F}v_1-\mathfrak{F}v_2))\\
&+\alpha\|\mathfrak{F}v_1-\mathfrak{F}v_2\|_{s+1}^2,
\end{align*}
where 
$$F_z(v)=\frac{1}{2}(z_\alpha+v)^2.$$ 
We show easily that
$$\|D^s(F_z(v_1)-F_z(v_2))\|^2\leq C(s)\|v_1-v_2\|_s^2(\|v_1+v_2\|_s^2+\|z_\alpha\|_s^2).$$

This allows us to get that
\begin{align*}
\frac{1}{2}\frac{d}{dt}\|\mathfrak{F}v_1-\mathfrak{F}v_2\|_s^2 &+\frac{\alpha}{2}\|\mathfrak{F}v_1-\mathfrak{F}v_2\|_{s+1}^2\leq \frac{C(s)}{\alpha}\|v_1-v_2\|_s^2(\|v_1+v_2\|_s^2+\|z_\alpha\|_s^2)\\
&\leq e^{\frac{t}{T}}\frac{C(s)(4R^2+\|z\|_{\Lambda_T(s)}^2)}{\alpha}\|v_1-v_2\|_{\Lambda_T(s)}^2.
\end{align*}
After integration in $t$, we find
\begin{align*}
\|\mathfrak{F}v_1-\mathfrak{F}v_2\|_s^2 &+\alpha\int_0^t\|\mathfrak{F}v_1-\mathfrak{F}v_2\|_{s+1}^2ds\leq Te^{\frac{t}{T}}\frac{C(s)(4R^2+\|z_\alpha\|_{\Lambda_T(s)}^2)}{\alpha}\|v_1-v_2\|_{\Lambda_T(s)}^2.
\end{align*}
We multiply this inequality by $e^{-\frac{t}{T}}$, the  $T$ found in the first step can be decreased if necessary to give a contraction.
\end{itemize}
We conclude by using the fixed point theorem.
\end{proof}
\begin{rmq}
By definition, $v$ is $\sigma(u_0,\mathcal{F}_t)$- adapted. Then the process $u=v+z_\alpha$ is continuous and $\sigma(u_0,\mathcal{F}_t)-$ adapted. Thanks to Lemma $\ref{intro_lemme_prog_mes}$, the process $u$ is progressively measurable with respect to that filtration.
\end{rmq}
\begin{proof}[End of the proof of Proposition $\ref{wellposedness}$, the wellposedness of $(\ref{equ})$.]
 Let $u_1$ and $u_2$ be two solutions of $(\ref{BOBf})$ starting respectively at $u_{1,0}$ and $u_{2,0}$, and set $w=u_1-u_2$, then the problem solved by $w$ is
\begin{align*}
\left\{\begin{array}{l r c}
\dt w+(H-\alpha)\dy w+w\dx w+\dx(wu_2)=0,\\
w|_{t=0}=u_{1,0}-u_{2,0}=:w_0.
\end{array}\right.
\end{align*}
Using the arguments of the proof of $(\ref{estim_s>2})$, we show
\begin{align*}
\sup_{t\in[0,T]}\|w(t)\|_s^2+\alpha\int_0^T\|w(r)\|_{s+1}^2dr\leq C(\alpha,T,\|\dx w\|_{L^\infty(0,T;L^\infty)},\|u_2\|_{L^\infty(0,T;H^s)})\|w_0\|_s^2\ .
\end{align*}
Hence follow the uniqueness and the continuity with respect to initial data. 
\end{proof}
The stochastic wellposedness that we just established combined with the estimates $(\ref{est_norm_inf2})$ and $(\ref{estim_s>2})$ implies the following result.
\begin{prop}
Let $j\geq 2$. Suppose $A_j$ finite. Then the equation $(\ref{BOBf})$ is well-structured on the Gelfand triple $(H^{j-1},H^j,H^{j+1})$ in the sense of Definition $\ref{intro_WS}$.
\end{prop}
\subsection*{Probabilistic estimates and Proof of Lemma $\ref{nem_central}$}

\paragraph{Exponential control of the $L^2-$norm.}
\begin{prop}
Let $p\geq1.$ Then the functional $E_0^p(u)=\|u\|^{2p}$ satisfies the conditions of Theorem $\ref{intro_KS_Ito_theorem}$ on the Gelfand triple $(H^{-1},L^2,H^1).$
\end{prop}
\begin{proof}
Thanks to the polynomial nature of $E_0^p(u)$ on $L^2,$ the uniform continuity on bounded sets and the conditions $(\ref{intro_KS_Ito_F1})$ and $(\ref{intro_KS_Ito_F2})$ follow easily. We confine ourself to the proof of $(\ref{intro_KS_Ito_F3}).$ The argument we use, to this end, is the following: As we have already shown, the solution of $(\ref{BOBf})$ can be represented as the sum of a "linear part" and a "nonlinear part." Now we will show that the nonlinear part can be controlled by the initial datum and an "exponential of the averaged linear part". On the other hand, we show that the linear part is exponentially controlled, then we get the needed control on the initial solution $u$. \\
\textbf{Control of the nonlinear part $v$.}
In this part we prove that for all $r,\epsilon >0$ and $p\geq 1$
\begin{equation}\label{controoole}
\|v(r)\|^{2p}\leq e^{f(r,\epsilon,p)}e^{\frac{2\epsilon p}{r}\int_0^r\|\dx z_\alpha\|_{L^\infty}^2ds}\left(\|u_0\|^{2}+\int_0^r\|z_\alpha\|_1^{4}ds\right)^p,
\end{equation}
where $f(r,\epsilon,p)=\frac{p}{4}\left(2r+\frac{r^2}{\epsilon}\right).$ Indeed, multiplying  $(\ref{equ_nonlin})$ by $v$ and integrating in $x$, one obtains
\begin{align*}
\frac{1}{2}\frac{d}{dt}\|v\|^2+\alpha\|v\|_1^2 &=-(v,\dx (vz_\alpha))-(v,z_\alpha\dx z_\alpha)\\
&=\frac{1}{2}[(v,v\dx z_\alpha)+(v,\dx z^2_\alpha)]\\
&\leq \frac{1}{2}[\|v\|\|vz_\alpha\|+\|v\|\|z_\alpha\|_1^2]\\
&\leq \frac{r}{8\epsilon}\|v\|^2+\frac{\epsilon}{r}\|v\|^2\|\dx z_\alpha\|_{L^\infty}^2+\frac{1}{4}\|v\|^2+\frac{1}{4}\|z_\alpha\|_1^4.
\end{align*}
Then we use the Gronwall lemma, choose $t=r$ and take the resulting inequality to the power $p$ to arrive at the claim.\\
\textbf{Exponential control of the linear part.}
Now, the linear part of the solution satisfies the estimate
\begin{equation}\label{control_exp_z}
\E e^{\frac{\epsilon}{t}\int_0^t\|z_\alpha\|_{2}^2ds}\leq 3,
\end{equation}
where $\epsilon>0$ is small enough. Indeed, by applying the Itô formula to $\|z\|_2^{2p}$ for $p\geq 1$  we have
\begin{equation}\label{chap4_estim_power,z}
\E\|z_\alpha\|_2^{2p}\leq \frac{A_1^p p^{p}}{\kappa^{p}}.
\end{equation}
Integrating in $t$, we find
\begin{equation*}
\E\left(\frac{1}{t}\int_0^t\|z_\alpha\|_{2}^{2p}ds\right)\leq \frac{A_1^p p^{p}}{\kappa^{p}}.
\end{equation*}
Thanks to Jensen's inequality, we infer
\begin{equation*}
\E\left(\frac{1}{t}\int_0^t\|z_\alpha\|_{2}^{2}ds\right)^p\leq \frac{A_1^p p^{p}}{\kappa^{p}}.
\end{equation*}
Now, let $0<\epsilon\leq \kappa /(2A_1e),$ then we have
\begin{equation*}
\E\frac{\left(\frac{\epsilon}{t}\int_0^t\|z_\alpha\|_{2}^{2}ds\right)^p}{p!}\leq \frac{p^{p}}{2^{p}e^pp!}.
\end{equation*}
We recall that for any integer $p>0$, we have that $p!\geq\left(\frac{p}{e}\right)^p,$ then we arrive at the claimed result.

\paragraph{Control of the quadratic variation of $E_0^p(u)$.}
We have that
\begin{align*}
\sum_{m\geq 0}a_m^2 &\E\int_0^t|\du (E_0^p)(u,e_m)|^2ds\ \lleq_p \sum_{m\in\Z}a_m^2\E\int_0^t\|u\|^{4(p-1)}|(u,e_m)|^2ds\\
&\ \lleq_p\ \E\int_0^t\|u\|^{4p-2}ds\ \lleq_p\ \E\int_0^t(\|v\|^{4p-2}+\|z_\alpha\|^{4p-2})ds.
\end{align*}
Set $q=4p-2,$ one sees, with the use of the estimate $(\ref{chap4_estim_power,z})$ (or just by invoking the Fernique theorem), that
\begin{equation*}
\E\int_0^t\|z\|_{2}^{q}ds<\infty \ \ \ for\ \ any \ \ t\geq 0.
\end{equation*}
Now we use the estimate $(\ref{controoole})$, then, for any $\epsilon>0,$ 
\begin{equation*}
\E\int_0^t\|v_s\|^{q}ds\leq \int_0^t e^{f(s,\epsilon,q)}\E\left[e^{\frac{\epsilon q}{s}\int_0^s\|\dx z_\alpha\|_{L^\infty}^2dr}\left(\|u_0\|^{2}+\int_0^s\|z_\alpha\|_{L1}^4dr\right)^q\right]ds.
\end{equation*}
Then for any $\delta>0$, we use the Young inequality to find 
\begin{equation*}
\E\int_0^t\|u_s\|^{q}ds\ \lleq \int_0^t e^{f(s,\epsilon,q)}\E\left[e^{\frac{q(1+\delta)\epsilon}{\delta s}\int_0^s\|\dx z_\alpha\|_{L^\infty}^2dr}+\underbrace{\left(\|u_0\|^{2}+\int_0^s\|z_\alpha\|_{1}^4dr\right)^{q(1+\delta)}}_{R_{q,\delta(s)}}\right]ds.
\end{equation*}
One uses the estimate $(\ref{chap4_estim_power,z})$ to bound $\E R_{q,\delta}(s)$
by $C_{q,\delta}(1+s^{q(1+\delta)})$. On the other hand, for any $\delta>0$ we choose $\epsilon>0$ small enough so that one can use the estimate $(\ref{control_exp_z})$ and the embedding $H^2\subset L^\infty$ to get the bound
\begin{equation*}
\E e^{\frac{2p(1+\delta)\epsilon}{\delta s}\int_0^s\|\dx z_\alpha\|_{L^\infty}^2dr}\leq 3.
\end{equation*} 
Then we get
\begin{equation*}
\E\int_0^t\|v_s\|^{2p}ds\ \lleq \int_0^te^{f(s,\epsilon,p)}(1+s^{q(1+\delta)})ds <\infty \ \ \ for \ \ all\ \ t\geq0.
\end{equation*}
\end{proof}

\begin{prop}\label{Ito_on_H}
Let $u$ be the solution of $(\ref{equ})$.
\begin{enumerate}
\item Suppose that $\E E_0(u_0)<\infty,$ then
\begin{equation}\label{itoeo}
\E E_0(u)+2\alpha\int_0^t\E\|u(s)\|_1^2ds=\E E_0(u_0)+\alpha A_0t.
\end{equation}
\item
 Let $p>1$. Suppose that $\Bbb{E}E_0^p(u_0)<\infty$, then 
\begin{equation}\label{itoeomoy_p}
\E E_0^p(u)\leq e^{-p\alpha t}\E E_0^p(u_0)+p^pA_0^p.
\end{equation}
\end{enumerate}
\end{prop}

\begin{proof}
The identity $(\ref{itoeo})$ is easily proven by applying the Itô formula to the conservation law $E_0(u).$ Let us prove $(\ref{itoeomoy_p}):$

For $p>1$, we apply the Itô formula to $E_0^p(u)$ to find
$$dE_0^p(u)=pE_0^{p-1}(u)dE_0(u)+\frac{\alpha p(p-1)}{2}E_0^{p-2}(u)\sum_{m\in\Z}\lambda_m^2|E'_0(u,e_m)|^2dt.$$
Taking the expectation, we get
$$\E E_0^p(u)+\E\int_0^tf_\alpha(u(s))ds= \E E_0^p(u_0),$$
where 
$$f_\alpha(u)=2p\alpha E_0^{p-1}(u)\|u\|_1^2-\alpha pE_0^{p-1}(u)A_0-\frac{\alpha p(p-1)}{2}E_0^{p-2}(u)\sum_{m\in\Z}\lambda_m^2|E'_0(u,e_m)|^2.$$
Let us set
\begin{align*}
Q=pE_0^{p-1}(u)A_0+\frac{p(p-1)}{2}E_0^{p-2}(u)\sum_{m\in\Z}\lambda_m^2|E'_0(u,e_m)|^2.
\end{align*}
Remarking  that 
$$\sum_{m\in\Z}\lambda_m^2|E'_0(u,e_m)|^2\leq 2A_0E_0(u),$$
we get, with the use of the Young inequality, the estimate
$$Q\leq \epsilon E_0^p(u)+\frac{p^{2p}}{\epsilon^{p-1}} A_0^p.$$
On the other hand
$$p\alpha E_0^{p-1}(u)\|u\|_1^2\geq p\alpha E_0^{p}(u).$$
Choosing $\epsilon=p$, we see that
$$\E f_{\alpha}(u)\geq p\alpha \E E_0^p(u)-p^{p+1}A_0^p\alpha.$$
Then
$$\E E_0^p(u)+p\alpha\int_0^t\E E_0^p(u(s))ds\leq\E E_0^p(u_0)+p^{p+1}A_0^p\alpha t.$$
Gronwall's lemma gives the claimed result.
\end{proof}

\paragraph{Control of higher order Sobolev norms.}
The polynomial nature of the Benjamin-Ono conservation laws $E_j$ allows to establish the following result:
\begin{prop}
Let $j\geq 1$, then the functional $E_j$ satisfies the conditions $(\ref{intro_KS_Ito_F1})$ and $(\ref{intro_KS_Ito_F2})$ of Theorem $\ref{intro_KS_Ito_theorem}$ on the triple $(H^{j-1},H^j,H^{j+1})$.
\end{prop}
In view of this result the "stopping time" version of the Itô formula $(\ref{intro_KS_Ito})$ applies to the functionals $E_j$.
\begin{thm}\label{ito_sur_tous} 
Let $j\geq 1$ be an integer. Suppose $A_j$ is finite. There are $\theta_j>0$, $\gamma_j>0$ such that for any solution $u$ of $(\ref{equ})$ in $H^2$  issued from $u_0\in H^2$ which satisfies $\E E_j(u_0)<\infty$, we have
\begin{align}\label{estimate_sur_loi_entier}
\E E_j(u)+\alpha\int_0^t\E\|u\|_{j+1}^2ds &\leq \E E_j(u_0)+\alpha A_j\left(t+c_j\int_0^t\E\|u\|_j^2ds+\gamma_j\int_0^t\E \|u\|(1+\|u\|)^{\theta_j}ds\right),
\end{align}
where $c_j$ depends only on $j$.
\end{thm}

\begin{proof}

The fact that $E_j(u)$ is preserved by the BO equation translates into
\begin{align*}
\du E_j(u,-H\dy u-u\dx u)=0.
\end{align*}
Setting the Markov time $\tau_n=\inf\{t\geq 0, \|u(t)\|_j>n\}$ and applying the Itô formula $(\ref{intro_KS_Ito}),$ we get
\begin{align*}
E_j(u(t\wedge\tau_n)) =E_j(u_0)&+\alpha\int_0^{t\wedge\tau_n}\left(\du E_j(u,\dy u)+\frac{1}{2}\sum_{m\in\Z}\lambda_m^2 \partial_u^2E_j(u,e_m)\right)ds \\ 
&+\sum_{m\in\Z}\lambda_m\int_0^{t\wedge\tau_n}\partial_uE_j(u,e_m)d\beta_m(s) .
\end{align*}
Then by the Doob optional stopping theorm, Theorem $\ref{intro_Doob}$, we have
\begin{equation*}
\E E_j(u(t\wedge\tau_n))=\E E_j(u_0)+\alpha\E\int_0^{t\wedge\tau_n}\left(\du E_j(u,\dy u)+\frac{1}{2}\sum_{m\in\Z}\lambda_m^2 \partial_u^2E_j(u,e_m)\right)ds. 
\end{equation*}
Using the monotone convergence theorem, we arrive at
\begin{equation*}
\E E_j(u(t))=\E E_j(u_0)+\alpha\E\int_0^{t}\left(\du E_j(u,\dy u)+\frac{1}{2}\sum_{m\in\Z}\lambda_m^2 \partial_u^2E_j(u,e_m)\right)ds. 
\end{equation*}
By Lemma $\ref{lemme_central},$ we have
\begin{align}
\du E_j(u,\dy u)\leq -\|u\|_{j+1}^2+P_j(\|u\|),\label{loc_1:ito_tous}
\end{align}
where $P_j$ is the polynomial of Lemma $\ref{lemme_central}.$
Following the arguments of the proof of Lemma $\ref{lemme_central},$ we establish that
\begin{equation}
|\partial_u^2 E_j(u,e_m)|\leq c_jm^{2j}(\|u\|_j^2+Q_j(\|u\|)),\label{loc_2:ito_tous} 
\end{equation}
where $Q_j(r)=q_jr(1+r)^{k_j},$ $q_j$ and $k_j$ depend only on $j$. 
Then take the expectation and combine $(\ref{loc_1:ito_tous})$ with $(\ref{loc_2:ito_tous})$ to get the claim.
\end{proof}

Now we are able to give the proof of Lemma $\ref{nem_central}$.
\begin{proof}[Proof of Lemma $\ref{nem_central}$]
Let $u$ be a stationary solution to $(\ref{BOBf})$ which satisfies the integrability assumption $(\ref{cond_p})$, and suppose that $A_j$ is finite for any $j$. Recall the estimate
\begin{equation}\label{induc_pp}
\E E_j(u)\leq \E \|u\|_j^2+c_n^+\E\|u\|^{2j+2}.
\end{equation}
Then using the integrability assumption $(\ref{cond_p})$, we see that $\E E_j(u)$ is finite as soon as $\E \|u\|_j^2<\infty.$\\
Note that, by the stationarity of $u$, the estimates $(\ref{estimate_sur_loi_entier})$ become (under the assumption that $\E E_j(u)$ is finite)
\begin{equation}\label{induc_prop}
\E\|u\|_{j+1}^2 \leq  A_j\left[1+c_j\E\|u\|_j^2+\gamma_j\E \|u\|(1+\|u\|)^{\theta_j}\right]
\end{equation}
since the distribution do not depend on $t$. We are going to argue by induction. Note that the needed induction property is given by the combination of $(\ref{induc_prop})$ and $(\ref{induc_pp})$ because they give at the same time the finiteness of $\E E_j(u)$ and the control of $\E \|u\|_{j+1}^2$ as soon as $\E \|u\|_j^2$ is finite. Moreover if $(\ref{cond_p})$ holds uniformly in $\alpha$ then so does $\E\|u\|_{j+1}^2$ once the control on $\E\|u\|_j^2$ is uniform in $\alpha$. It remains to prove the initial step, namely $\E \|u\|_1^2$ is finite and does not depend on $\alpha$. But using again the integrability assumption at the order $p=2,$ the stationarity of $u$ combined with the estimate $(\ref{itoeo})$ gives
\begin{align*}
\E\|u\|_1^2=\frac{A_0}{2}.
\end{align*}
\end{proof}

\section{Stationary measures for the viscous problem}\label{section_inv_visc}
Consider the stochastic BOB problem $(\ref{BOBf})$ posed on $\dot{H}^2(\T).$ By the estimates $(\ref{itoeo})$, $(\ref{itoeomoy_p})$ and Theorem $\ref{ito_sur_tous}$, we have

\begin{align}
\E E_0(u) +2\alpha\int_0^t\E\|u\|_1^2ds &=\E E_0(u_0)+\alpha A_0t,\nonumber\\
\E E_0^p(u) &\leq e^{-p\alpha t}\E E_0^p(u_0)+C_pA_0^p, \label{secondestimate}\\
\E E_1(u) +\alpha\int_0^t\E\|u\|_2^2ds &\leq \E E_1(u_0)+\alpha\left(A_1t+c_1\int_0^t\E\|u\|_1^2ds+\int_0^t\E W_1(\|u\|)ds\right),\nonumber\\
\E E_2(u) +\alpha\int_0^t\E\|u\|_3^2ds &\leq \E E_2(u_0)+\alpha\left(A_2t+c_2\int_0^t\E\|u\|_2^2ds+\int_0^t\E W_2(\|u\|)ds\right),\nonumber
\end{align}
where $W_1$ and $W_2$ are the polynomials used in the estimate $(\ref{estimate_sur_loi_entier})$, their expectation is controlled using $(\ref{secondestimate})$. Now suppose $u_0=0$ almost surely, then by an induction argument, we get

\begin{equation*}
\E E_2(u)+\alpha\int_0^t\E\|u\|_3^2ds\leq \alpha Ct,
\end{equation*}
where $C$ is universal. Now in view of Remark $\ref{remark_nouvelles_lois}$, we can suppose $E_n(u)\geq0$ (indeed, adding $c\|u\|^6$ to $E_2(u)$ we find a similar estimate). Then
\begin{equation}
\frac{1}{t}\int_0^t\E\|u\|_3^2ds\leq C,\label{moy_en_t_BK}
\end{equation}
where $C$ is, in particular, independent of $t$. Denote by $\lambda_\alpha(t)$ the law of the solution $u(t)$ to $(\ref{BOBf})$ starting at $0$, and consider the time average
\begin{equation*}
\bar{\lambda}_\alpha(t)=\frac{1}{t}\int_0^t\lambda_\alpha(s)ds.
\end{equation*}
Using the estimate $(\ref{moy_en_t_BK})$, we show 
\begin{equation}\label{locc}
\int_{H^2}\|u\|_3^2\bar{\lambda}_\alpha(t)(du)\leq C.
\end{equation}
Then by the Chebyshev inequality we have
\begin{align*}
\bar{\lambda}_\alpha(t)(\{\|u\|_3>R\})\leq \frac{C}{R^2} \ \ \ for\ any \ R>0.
\end{align*}
Thus the compactness of the embedding $H^3(\T)\subset H^2(\T)$ combined with the Prokhorov theorem
implies that the family $\{\lambda_\alpha(t),\ t>0\}$ is compact with respect to the weak topology of $H^2.$ Then for any $\alpha$ we denote by $\mu_\alpha$ an accumulation point at infinity of the above family. The classical Bogoliubov-Krylov argument implies that $\mu_\alpha$ is a stationary measure for $(\ref{BOBf})$. Passing to the limit $t\to\infty$ in $(\ref{locc})$ (using an approximation argument), we see that $\mu_\alpha(H^3)=1$ for any $\alpha.$ We summarize these results in the following statement:
\begin{prop}
For any $\alpha\in (0,1)$, the stochastic BOB equation $(\ref{BOBf})$ posed in $H^2(\T)$ has a stationary measure $\mu_\alpha$ concentrated on $H^3(\T).$
\end{prop}
\begin{thm}\label{thm_reg_n,p_alpha}
Let $\alpha\in (0,1).$ Suppose that $A_n$ is finite for any $n.$ Then any stationary measure $\mu_\alpha$ of the problem $(\ref{BOBf})$ posed in $\dot{H}^2(\T)$ satisfies
\begin{align}
\int_{H^2(\T)}\|u\|_1^2\mu_\alpha(du) &=\frac{A_0}{2}, \label{ident_esp_visc_h1}\\
\int_{H^2(\T)}\|u\|^{2p} \mu_\alpha(du)&\leq p^p A_0^p \label{estim_esper_visc_moy_lp}\ \ \ \text{for any $1\leq p<\infty$},\\
\int_{H^2(\T)}\|u\|_{n}^2\mu_\alpha(du) &\leq D_n \ \ \ \text{for any $n\geq 2$} \label{estim_esper_viscn>2}, 
\end{align}
where, for any $n$, $D_n$ does not depend on $(t,\alpha)$.
\end{thm}
\begin{proof}
It suffices to prove $(\ref{estim_esper_visc_moy_lp})$ since, then the estimate $(\ref{estim_esper_viscn>2})$ follows from Lemma $\ref{nem_central}.$ We combine $(\ref{itoeo})$ and the stationarity of $u$ to get $(\ref{ident_esp_visc_h1}).$ Let us prove $(\ref{estim_esper_visc_moy_lp})$.  \\
For this end, let $R>0,$ consider a $C^\infty$-function $\chi_R$ 
satisfying
\begin{equation*}
\chi_R(u)=\left\{\begin{array}{l r c}
1,\ \ \ \text{if $\|u\|_2\leq R,$}\\ 0,\ \ \ \text{if $\|u\|_2> R+1.$}
\end{array}
\right.
\end{equation*}

Let $p\geq 1,$ we have
\begin{equation}
\int_{H^2}E_0^p(u)\chi_R(u)\mu_\alpha(du)=\int_{H^2}\E\{E_0^p(u(t,v))\chi_R(u(t,v))\}\mu_\alpha(dv),\label{pour_finitude}
\end{equation}
where $u(.,v)$ is the solution of $(\ref{BOBf})$ starting at $v$. We pass to the limit $t\to\infty$ in the right hand side of $(\ref{pour_finitude})$ and using $(\ref{itoeomoy_p})$ ($u$ is in the ball of size $R$) and the stationarity of $\mu_\alpha$, we find
\begin{equation*}
\int_{H^2}E_0^p(u)\chi_R(u)\mu_\alpha(du)\leq p^pA_0^p.
\end{equation*}
Now Fatou's Lemma allows to conclude.
\end{proof}

\begin{cor}\label{cor_tt_mes_est_reg}
Let $\alpha\in (0,1)$. Suppose $A_n<\infty$ for any $n.$ Then any stationary measure $\mu_\alpha$ for  the stochastic BOB problem $(\ref{BOBf})$ posed in $\dot{H}^2(\T)$ is concentrated on $C^\infty(\T).$
\end{cor}
\begin{proof} Let $n>2$.
Combining the estimate $(\ref{estim_esper_viscn>2})$ and the Chebyshev inequality we find
\begin{equation*}
\mu_\alpha(\{u\in H^2:\ \ \|u\|_n\geq R\})\leq \frac{D_n}{R^2}.
\end{equation*}
Setting $B_n(0,R)$ the ball in $H^n$ of center $0$ and radius $R$, we have
\begin{align*}
\int_{H^2}\Bbb 1_{B_n(0,R)}(u)\mu_\alpha(du)=\mu_\alpha(B_n(0,R))\geq 1-\frac{D_n}{R^2}.
\end{align*}
Passing to the limit on $R$ (with the use of the Lebesgue convergence theorem), we get
\begin{equation*}
\mu_\alpha(H^n(\T))=1.
\end{equation*}
Thus
\begin{equation*}
1=\mu_\alpha(\cap_{n>2}H^n(\T))=\mu_\alpha(C^\infty(\T)).
\end{equation*}
\end{proof}
\section{Invariant measure for the BO equation}

In this section, $S_t:H^3(\T)\rightarrow H^3(\T)$, $t\geq 0,$ denotes the flow of the Benjamin-Ono equation $(\ref{BO}).$ The map $S_{t,\alpha}:H^3\rightarrow H^3$ denotes the one of the stochastic Benjamin-Ono-Burgers equation $(\ref{BOBf})$. We denote by $\phi_t,\ \phi^*_t, \ \phi_{t,\alpha}, \ \phi^*_{t,\alpha}$ the associated Markov semi-groups, respectively. We suppose in what follows that $A_n<\infty$ for any $n>0.$
\subsection*{Some convergence results of stochastic BOB equation to the BO equation}

\begin{nem}\label{pointwise_conver}
For any $T>0$. For any $w\in H^3(\T)$, we have, $\P$-almost surely,
\begin{equation*}
\sup_{t\in[0,T]}\|S_{t,\alpha}w-S_tw\|_2\to 0\ \ \ as\ \ \ \alpha\to 0.
\end{equation*}
\end{nem}

\begin{proof}
We write 
\begin{align*}
\|S_{t,\alpha}w-S_tw\|_2=\|v+z_\alpha-S_tw\|_2\leq\|v-S_tw\|_2+\|z_\alpha\|_2,
\end{align*}
where
\begin{equation*}
z_\alpha(t)=\sqrt{\alpha}\int_0^te^{-(t-s)(H-\alpha)\dy}d\zeta(s)=\sqrt{\alpha}z(t)
\end{equation*}
and $v$ is the solution of
\begin{align}
\dt v &+H\dy v+(v+z_\alpha)\dx (v+z_\alpha)=\alpha\dy v \label{nonlin_equ_2}\\
v_{t=0} &=w.
\end{align}
Thanks to the estimate $(\ref{estim_sup_z})$, we have that $\sup_{t\in[0,T]}\|z_\alpha\|_2=\sqrt{\alpha}\sup_{t\in[0,T]}\|z\|_2$, where the quantity\\ $\sup_{t\in[0,T]}\|z\|_2$ does not depend on $\alpha$. Setting $h=v-S_tw$, we have
\begin{equation*}
\sup_{t\in[0,T]}\|S_{t,\alpha}w-S_tw\|_2\leq \sup_{t\in[0,T]}\|h\|_2+\sqrt{\alpha}\sup_{t\in[0,T]}\|z\|_2.
\end{equation*}
We claim that $\sup_{t\in[0,T]}\|h\|_2=O(\sqrt{\alpha}).$ Indeed using the estimate $(\ref{estim_s>2})$ and the $H^3$-conservation law, we show that

\begin{equation*}
\|h\|_2^3\leq c\|h\|\|h\|_3^{2}\leq C(T,\|w\|_{L^\infty(0,T;H^3)},\|z\|_{H^3})\|h\|.
\end{equation*}
Taking the difference between $(\ref{nonlin_equ_2})$ and the BO equation $(\ref{BO})$, we see that $h$ satisfies
\begin{equation*}
\dt h+H\dy h+h\dx h=-\dx (hS_tw)-\dx(vz_\alpha)-z_\alpha\dx z_\alpha.
\end{equation*}
We multiply the above equation by $h$ and we integrate on $\T$ to get
\begin{align*}
\dt \|h\|^2=\frac{1}{2}(h^2,\dx S_tw)-(h,\dx(vz_\alpha))-\frac{1}{2}(h,\dx z_\alpha^2).
\end{align*}
By the Cauchy-Schwarz inequality and the algebra structure of $H^1$ we find
\begin{align*}
\dt\|h\|^2 &\leq \frac{1}{2}\|h\|^2\|\dx S_tw\|_{L^\infty}+\frac{1}{2}\|h\|^2+C\|v\|_1^2\|z_\alpha\|_1^2+\frac{1}{4}\|h\|^2+\frac{1}{4}\|z_\alpha\|_1^4\\
&\leq \frac{1}{2}\|h\|^2(\|\dx S_tw\|_{L^\infty}+\frac{3}{2})+C\alpha\sup_{t\in[0,T]}\|v\|_1^2\sup_{t\in[0,T]}\|z\|_1^2+\frac{\alpha^2}{4}\sup_{t\in[0,T]}\|z\|_1^4
\end{align*}
Using the $H^2$-conservation law, we control $\|S_tw\|_{L^\infty(0,T;H^{3/2+})}$ (which does not depend on $\alpha$) and $\|v\|_{L^\infty(0,T;H^{1})}$ (see the estimate $(\ref{est_norm_inf2})$). It remains to apply the Gronwall lemma to get the claim.
\end{proof}

\begin{nem}\label{uniform_conve_equat}
For all $T,R,r>0$, we have
\begin{equation*}
\sup_{w\in B(0,R)}\sup_{t\in[0,T]}\E\left[\|S_{t,\alpha}w-S_tw\|_2\Bbb 1_{\{\|z\|_{L^\infty(0,T;H^2)}\leq r\}}\right]=O_{R,r,T}(\sqrt{\alpha}).
\end{equation*}
Here $B(0,R)$ is the ball in $H^3(\T)$ of center $0$ and radius $R$.
\end{nem}

\begin{proof}
\begin{align*}
\E\left[\|S_{t,\alpha}w-S_tw\|_2\Bbb 1_{\{\|z\|_{L^\infty(0,T;H^2)}\leq r\}}\right] &=\int_\Omega\|S_{t,\alpha}w-S_tw\|_2\Bbb 1_{\{\|z\|_{L^\infty(0,T;H^2)}\leq r\}}(\omega)d\P(\omega)\\
&\leq \int_\Omega[\|h\|_2+r\sqrt{\alpha}]\Bbb 1_{\{\|z\|_{L^\infty(0,T;H^2)}\leq r\}}(\omega)d\P(\omega),
\end{align*}
where $h=v-S_tw$ as before. The arguments of the proof of Lemma $\ref{pointwise_conver}$ allow to see that $\sup_{t\in[0,T]}\|h\|_2\leq C_{R,r,T}\sqrt{\alpha}$. This gives the claimed result.
\end{proof}
\subsection*{An accumulation point for the viscous stationary measures}
In what follows we denote by $M(H^3)$ the space of probability measures on $H^3.$
\begin{thm}\label{thm_inv_mes}
For any sequence $(\alpha_k)_{k\in\N}\subset(0,1)$ converging to $0$ as $k\to\infty$, there is a subsequence $\alpha_{r(k)}$ and $\mu\in M(H^3)$ such that:
\begin{itemize}
\item $\lim_{k\to\infty}\mu_{\alpha_{r(k)}}=\mu$ \ \ in the weak topology of $H^3,$
\item $\mu$ is invariant under the flow of the Benjamin-Ono equation in $H^{3}(\T),$

\item $\mu$ is concentrated on $C^\infty(\T),$
\item $\mu$ satisfies 
\begin{align}
\int_{H^3(\T)}\|u\|^2_1\mu(du) &=\frac{A_0}{2}, \label{ident_esp_h1}\\
\int_{H^3(\T)}\|u\|^{2p}\mu(du)  &\leq p^p A_0^p \ \ \ \text{for any $1\leq p<\infty$,}\label{estim_esp_lp}\\
\int_{H^3(\T)}\|u\|_{n}^2\mu(du) &<\infty\label{est_esp_n>2}\ \ \ \text{for $n\geq 2$}.
\end{align}
\end{itemize}
\end{thm}

\begin{proof}
The proof consists in the following four steps:
\paragraph{$1.$ \textbf{Existence of an accumulation point $\mu$.}}
The estimate $(\ref{estim_esper_viscn>2})$ with $n=4$ implies the tightness of the sequence of measures $(\mu_\alpha)$ in $H^{3}(\T)$ and, by the Prokhorov theorem, the existence of the claimed accumulation point $\mu$  on $H^3(\T).$

\paragraph{$2.$ \textbf{Invariance of $\mu$ under the Benjamin-Ono flow.}}
Denote by $(\mu_{\alpha_k})_{k\in\N}$ a subsequence of $(\mu_\alpha)$ converging to $\mu$ (with $\lim_{k\to\infty}\alpha_k=0$), to simplify the notations we write $\mu_k$ instead. The corresponding flow and Markov semi-group will be denoted $S_{t,k}$ and $\phi_{t,k}.$\\
The following diagram represents the idea of proof of the invariance of $\mu$:
$$
\hspace{10mm}
\xymatrix{
  \phi_{t,k}^*\mu_{k} \ar@{=}[r]^{(I)} \ar[d]^{(III)} & \mu_{k} \ar[d]^{(II)} \\
    \phi_t^*\mu \ar@{=}[r]^{(IV)} & \mu
  }
$$
The equality $(I)$ is the invariance of $\mu_k$ by $\phi_{t,k}$, and $(II)$ is proved above. Then $(IV)$ is proved once $(III)$ is checked.

Let $f$ be a real bounded Lipschitz function on $H^2(\T)$. Without loss of generality assume that $f$ is bounded by $1$. Then
\begin{align*}
(\phi^*_{k,t}\mu_k,f)-(\phi^*_{t}\mu,f)&=(\mu_k,\phi_{t,k}f)-(\mu,\phi_tf)\\
&=(\mu_k,\phi_{t,k}f-\phi_tf)-(\mu-\mu_k,\phi_tf)\\
&=A-B.
\end{align*}
The term $B$ converges to $0$ as $k\to\infty$ by the weak convergence of $(\mu_k)$ to $\mu.$ And for any $R>0$
\begin{align*}
|A| &\leq \int_{H^3}\E|f(S_{t,k}w)-f(S_tw)|\mu_k(dw)\\
&=\int_{B(0,R)}\E|f(S_{t,k}w)-f(S_tw)|\mu_k(dw)+\int_{H^3\backslash B(0,R)}\E|f(S_{t,k}w)-f(S_tw)|\mu_k(dw)\\
&=A_1+A_2.
\end{align*}
Recalling that $f$ is bounded by $1$, we get by the Chebyshev inequality
\begin{align}
A_2\leq 2\mu_k(H^3\backslash B(0,R))\leq \frac{C}{R^2},\label{estim_A_2:pass_limit}
\end{align}
where $C$ is finite and does not depend on $k$ (estimate $(\ref{estim_esper_viscn>2})$). Denote by $L^\infty_tH^2_x$ the space $L^\infty(0,T;H^2).$ Let $r>0$, we have
\begin{align*}
A_1 &=\int_{B(0,R)}\E\left[|f(S_{t,k}w)-f(S_tw)|\Bbb 1_{\{\|z\|_{L^\infty_tH^2_x}\leq r\}}\right]\mu_k(dw)+\int_{B(0,R)}\E\left[|f(S_{t,k}w)-f(S_tw)|\Bbb 1_{\{\|z\|_{L^\infty_tH^2_x}> r\}}\right]\mu_k(dw)\\
&=A_{1,1}+A_{1,2}.
\end{align*}
As before, since $f$ is bounded by $1$, we use $(\ref{ito_sur_linear2})$ and Chebyshev's inequality to get
\begin{equation*}
A_{1,2}\leq\frac{C_T}{r^2}.
\end{equation*}
 On the other hand, $f$ is Lipschitz on $H^2$,  we have
\begin{align*}
A_{1,1} &\leq C_f\int_{B(0,R)}\E\left[\|S_{t,k}w-S_tw\|_2\Bbb 1_{\{\|z\|_{L^\infty_tH^2_x}\leq r\}}\right]\mu_k(dw)\\
&\leq C_f\sup_{w\in B(0,R)}\E\left[\|S_{t,k}w-S_tw\|_2\Bbb 1_{\{\|z\|_{L^\infty_tH^2_x}\leq r\}}\right],
\end{align*}
where $C_f$ is the Lipschitz constant of $f$.\\
According to Lemma $\ref{uniform_conve_equat}$, we find
\begin{equation*}
A_{1,1}\leq C_{f,R,r,T}\sqrt{\alpha_k}.
\end{equation*}
Finally, we arrive at
\begin{equation*}
|A|\leq C_{f,R,r,T}\sqrt{\alpha_k}+\text{Const}(T)\left(\frac{1}{r^2}+\frac{1}{R^2}\right),
\end{equation*}
where Const does not depend on $k$. We get the desired result after passing to the limits in this order 
\begin{align*}
k &\to\infty, \\ R,r &\to \infty.
\end{align*}

\paragraph{$3.$ \textbf{The estimates for the measure $\mu$.}}
 Denoting by $\chi_R$ a bump function on the ball $B(0,R)$ of $H^3(\T)$, by $(\ref{ident_esp_visc_h1})$ we have
\begin{align*}
\int_{H^3}\chi_R(v)\|v\|_1^2\mu_k(dv)&\leq \frac{A_0}{2}.
\end{align*}
Passing to the limit $k\to\infty$ we find
\begin{align*}
\int_{H^3}\chi_R(v)\|v\|_1^2\mu(dv)&\leq \frac{A_0}{2}.
\end{align*}
Then Fatou's lemma gives
\begin{align}\label{non_ident_h1_pass_lim}
\E\|u\|_1^2=\int_{H^3}\|v\|_1^2\mu(dv)&\leq \frac{A_0}{2}.
\end{align}
We proceed similarly to show $(\ref{estim_esp_lp})$ and $(\ref{est_esp_n>2}).$\\
Now we write
\begin{align*}
\frac{A_0}{2}&=\int_{B(0,R)}\|v\|_1^2\mu_k(dv)+\int_{H^3\backslash B(0,R)}\|v\|_1^2\mu_k(dv).
\end{align*}
We use the Cauchy-Schwarz and Chebyshev inequalities to show that
\begin{align*}
\int_{H^3\backslash B(0,R)}\|u\|_1^2\mu_k(du) &=\int_{H^3}\|u\|_1^2\Bbb 1_{\|u\|_3>R}(u)\mu_k(du)\\
&\leq \left(\int_{H^3}\|u\|^4_1\mu_k(du)\right)^{\frac{1}{2}}\left(\mu_k(\|u\|_3>R)\right)^{\frac{1}{2}}\\
&\leq \frac{\sqrt{\E[\|u\|_3^2]\E[\|u\|_1^4]}}{R}.
\end{align*}
We can control $\E[\|u\|_1^4]$ and $\E[\|u\|_3^2]$ uniformly in $k$ combining interpolation inequalities and the estimates $(\ref{estim_esp_lp})$ and $(\ref{est_esp_n>2}).$ 
Then there is a constant $C>0$ independent of $k$ such that
\begin{align*}
\frac{A_0}{2}-\frac{C}{R}\leq\int_{H^3}\chi_R(v)\|v\|_1^2\mu_k(dv).
\end{align*}
We find
$(\ref{ident_esp_h1})$ after passing to the limits in the order
\begin{align*}
k &\to\infty,\\ R &\to \infty,
\end{align*}
and combining this with $(\ref{non_ident_h1_pass_lim}).$

\paragraph{$4.$ \textbf{The measure $\mu$ is concentrated on $C^\infty(\T).$}}
This immediately follows from the estimates $(\ref{est_esp_n>2})$ with use of the arguments of the proof of Corollary $\ref{cor_tt_mes_est_reg}.$
\end{proof}
\section{Qualitative properties of the measure}

\subsection*{Absolute continuity of some observables with respect to the Lebesgue measure}
The following result is inspired by \cite{armen_nondegcgl,KS12} where the local time concept is used to deduce non-degeneracy properties of measures constructed for the nonlinear Schrödinger and Euler equations.
\begin{thm}\label{non_den_plus}
Suppose that $\lambda_m\neq 0$ for all $m$. Then for any integer $n\geq 1$, there are constants $b_n$ and $c_n$ such that the distribution of the observable $\tilde{E}_n(u):=E_n(u)+c_n\|u\|^2(1+\|u\|^2)^{b_n}$ under $\mu$ has a density w.r.t. the Lebesgue measure on $\R$.
\end{thm}

For the proof of Theorem $\ref{kuksin_thm}$ below, we refer the reader to \cite{armen_nondegcgl} and the proof of Theorem $5.2.12$ of \cite{KS12} where the authors prove similar results in the case  of the nonlinear Schrödinger and Euler equations respectively.
\begin{thm}\label{kuksin_thm}
The measure $\mu$ constructed in Theorem \ref{thm_inv_mes} satisfies the following non-degeneracy properties:
\begin{enumerate}
\item Let $\lambda_m\neq 0$ for at least two indices. Then $\mu$ has no atom at $0$ and
\begin{equation}\label{nongen_1}
\mu(\{u\in C^\infty:\ \|u\|\leq \delta\})\leq C
\sqrt{A_0}\gamma^{-1}\delta\ \ \ \text{for all $\delta>0$},
\end{equation}
where $\gamma=\inf\{A_0-\lambda_m^2,\ m\in\Bbb Z\}$ and $C$ is a universal constant.
\item Let $\lambda_m\neq 0$ for all indices. Then there is an increasing continuous function $h(r)$ vanishing at $r=0$ such that
\begin{equation}\label{nongen_2}
\mu(\{u\in C^\infty(\T):\ \|u\|\in \Gamma\})\leq h(\ell(\Gamma))
\end{equation}
for any Borel set $\Gamma\subset \R$, where $\ell$ stands for the Lebesgue measure on $\R.$
\end{enumerate}
\end{thm}

\begin{proof}[Proof of Theorem $\ref{non_den_plus}$]
We prove the claim for the stationary measures in the case $\alpha>0,$ with uniform bounds in $\alpha.$ Then we can pass to the limit $\alpha\to 0$ to obtain the desired result using the Portmanteau theorem. First we apply the Itô formula to $\tilde{E}_n(u)$:
\begin{equation*}
\tilde{E}_n(u(t))=\tilde{E}_n(u(0))+\alpha\int_0^tA(s)ds+\sqrt{\alpha}\sum_{m\in\Z}\lambda_m\int_0^t\tilde{E}_n'(u,e_m)d\beta_m(s),
\end{equation*}
where
\begin{equation*}
A(s)=\du \tilde{E}_n(u,\dy u)+\frac{1}{2}\sum_{m\in\Z}\lambda_m^2\partial_u^2\tilde{E}_n(u,e_m).
\end{equation*}
Denote by $\Lambda_t(a,\omega)$ its local time which reads (see the identity $(A.45)$ of \cite{KS12})
\begin{align*}
\Lambda_t(a,\omega) &=(\tilde{E}_n(u(t))-a)_+-(\tilde{E}_n(u(0))-a)_+-\alpha\int_0^tA(s)\Bbb 1_{(a,+\infty)}(\tilde{E}_n(u))ds \\ &-\sqrt{\alpha}\sum_{m\in\Z}\lambda_m\int_0^t\Bbb 1_{(a,+\infty)}(\tilde{E}_n(u))\tilde{E}_n'(u,e_m)d\beta_m(s).
\end{align*}
Using the stationarity of $u$, we infer that
\begin{equation}\label{id_1}
\E\Lambda_t(a)=-\alpha t\E [A(0)\Bbb 1_{(a,+\infty)}(\tilde{E}_n(u))].
\end{equation}
Now using the (local time) identity $A.44$ of \cite{KS12}  with the function $\Bbb 1_\Gamma,$ we get
\begin{equation*}
2\int_\Gamma \Lambda_t(a)da=\alpha\sum_{m\in\Z}\lambda_m^2\int_0^t\Bbb 1_\Gamma(\tilde{E}_n(u))\tilde{E}_n'(u,e_m)^2ds.
\end{equation*}
The stationarity of $u$ gives again
\begin{equation}\label{id_2}
2\int_\Gamma \E\Lambda_t(a)da=\alpha t\sum_{m\in\Z}\lambda_m^2\E[\Bbb 1_\Gamma(\tilde{E}_n(u))\tilde{E}_n'(u,e_m)^2].
\end{equation}
Comparing $(\ref{id_1})$ and $(\ref{id_2})$, we find
\begin{equation}\label{id_final}
\sum_{m\in\Z}\lambda_m^2\E[\Bbb 1_\Gamma(\tilde{E}_n(u))\tilde{E}_n'(u,e_m)^2]\leq 2\lambda(\Gamma)\E|A(0)|\leq C\ell(\Gamma).
\end{equation}
Recall now the form of $\tilde{E}_n(u):$
\begin{equation*}
\tilde{E}_n(u)=\|u\|_n^2+R_n(u)+P_n(\|u\|^2),
\end{equation*}
where
\begin{equation*}
P_n(r)=c_nr(1+r)^{b_n}.
\end{equation*}
Then
\begin{equation*}
\tilde{E}_n'(u,v)=2(D^nu,D^nv)+R_n'(u,v)+2(u,v)P_n'(\|u\|^2).
\end{equation*}
Recalling Remark $\ref{remark_nouvelles_lois}$, we have
\begin{equation}\label{minorationnn}
\tilde{E}_n'(u,u)\geq \|u\|_n^2.
\end{equation}
Now we define the operator $H_n$ so that
\begin{equation*}
\tilde{E}_n'(u,v)=(H_nu,v).
\end{equation*}
Therefore
\begin{align*}
(H_nu,u) &=\sum_{m\in\Z}u_m(H_nu,e_m)\\
&=\sum_{|m|\leq N}u_m(H_nu,e_m)+\sum_{|m|>N}u_m(H_nu,e_m)\\
&\leq \frac{\|u\|}{\underline{\lambda}_N}\left(\sum_{|m|\leq N}\lambda_m^2(H_nu,e_m)^2\right)^{\frac{1}{2}}+\|H_nu\|\left(\sum_{|m|>N} u_m^2\right)^{\frac{1}{2}}\\
&\leq \frac{\|u\|_1}{\underline{\lambda}_N}\left(\sum_{m\in\Z}\lambda_m^2\tilde{E}_n'(u,e_m)^2\right)^{\frac{1}{2}}+\|H_nu\|\frac{\|u\|_1}{N},
\end{align*}
where $\underline{\lambda}_N=\min\{\lambda_m,\ \ |m|\leq N\}>0$ for any $N>0$. We take into account $(\ref{minorationnn})$ and consider $u$ belonging to 
\begin{equation*}
K_\epsilon =\left\{v:\ \ \|v\|\geq \epsilon,\ \ \|H_nv\|\leq \frac{1}{\epsilon}\right\}.
\end{equation*}
We get
\begin{equation*}
\sum_{m\in\Z}\lambda_m^2\tilde{E}_n'(u,e_m)^2\geq \underline{\lambda}_N^2\left(\epsilon-\frac{1}{N\epsilon}\right)^2.
\end{equation*}
The integer $N$ can be chosen to depend on $\epsilon$ so that we have 
$$\alpha(\epsilon):=\underline{\lambda}_N^2\left(\epsilon-\frac{1}{N\epsilon}\right)^2>0.$$ 
Then, by $(\ref{id_final})$
\begin{equation*}
\mu(\{u:\ \tilde{E}_n(u)\in \Gamma\}\cap K_\epsilon)\leq \frac{C}{\alpha(\epsilon)}\ell(\Gamma).
\end{equation*}
Consider now the complementary set 
\begin{equation*}
K_\epsilon^c=\left\{u:\ \|u\|< \epsilon\ \ \ \text{or}\ \ \ \|H_nu\|>\frac{1}{\epsilon}\right\}
\end{equation*}
Since
\begin{equation*}
\E\|H_nu\|\leq \text{Const}.
\end{equation*}
Using the Chebyshev inequality, we find

\begin{equation*}
\mu_\alpha\left(\left\{u:\ \|H_nu\|> \frac{1}{\epsilon}\right\}\right)\leq \text{Const}\ \epsilon.
\end{equation*}
By Theorem $\ref{kuksin_thm}$, we have that
\begin{equation*}
\mu_\alpha(\{u:\ \|u\|<\epsilon\})\leq C\epsilon.
\end{equation*}
Finally we write
\begin{align*}
\mu_\alpha(\{u:\ \tilde{E}_n(u)\in\Gamma\})&\leq \mu(\{u:\ \tilde{E}_n(u)\in\Gamma\}\cap K_\epsilon)+\mu(K_\epsilon^c)\\ &\leq \frac{C_1}{\alpha(\epsilon)}\ell(\Gamma)+C_2\epsilon.
\end{align*}
This, combined with the Portmanteau theorem, proves the absolute continuity of $\tilde{E}_n(u)$ under $\mu$ w.r.t. the Lebesgue measure on $\R.$
\end{proof}

\subsection*{About the dimension of the measure $\mu$}

This subsection is inspired by \cite{kuksin_nondegeul,KS12} where it was proved that the invariant measures constructed for the Euler equation are not concentrated on a countable union of finite-dimensional compact sets. The proof relies on a Krylov estimate (see section $A.9$ of \cite{KS12} and the original paper \cite{kry}) for Itô processes. Roughly speaking, this estimate provides an inequality of the type $(\ref{id_final})$ for multi-dimensional processes. Namely, for a $d-$dimensional stationary Itô process 
\begin{equation*}
y_t=y_0+\int_0^tx_sds+\sum_{j=1}^\infty\int_0^t\theta_j(s)d\beta_j(s),
\end{equation*}
define the non-negative $d\times d-$matrix $\sigma$ with entries
\begin{equation*}
\sigma_{m,n}=\sum_{j=1}^\infty\theta_j^m\theta_j^n,
\end{equation*}
where $\theta_j^i$ is the $i$-th component of the $d$-vector $\theta_j.$
Let $f:\R^d\to\R$ be a bounded measurable function. Then the Krylov estimate is 
\begin{equation}\label{Krylov-est}
\E\int_0^1f(y_t)(\det\sigma_t)^{1/d}dt \leq C_d|f|_d\E\int_0^1|x_t|dt,
\end{equation}
where $|.|_d$ stands for the $L^d-$norm and $C_d$ is a constant that only depends on $d$.

In our context the independence needed to make the Krylov estimate successful leads to solving nonlinear differential equations with order increasing with the size of the underlying vector (process). This is due to the structure of the BO conservation laws and represents a technical difficulty as discussed in the introduction, while in the Euler case the components of this vector can be chosen to satisfy this independence. We bypass the equation mentioned above in the $2$-dimensional case by splitting suitably the phase space. 

\begin{thm}\label{Hausd}
The measure $\mu$ is  of at least $2$-dimensional nature in the sense that any compact set of Hausdorff dimension smaller than $2$ has $\mu$-measure $0$.
\end{thm}

Before proving Theorem $\ref{Hausd}$, we describe the general framework.

We use the following splitting of $H^2(\T)$:
\begin{equation*}
H^2(\T)=O\cup O^c,
\end{equation*}
where
\begin{equation}\label{set_of_preserv}
O:=\left\{u:\ \ \int u^2H\dy u=0\right\}.
\end{equation}
 Consider the functionals on $\dot{H}^1(\T)$ defined by
\begin{equation*}
F_j(u)=\frac{1}{j+1}\int u^{j+1}, \ \ j=1,2.
\end{equation*}
Remark that $F_1$ is preserved by the BO equation. Now for $u$ a solution of $(\ref{BO})$, we have that 
\begin{equation*}
\dt F_2(u)=0 \ \ \ \text{on $O$.}
\end{equation*} 
Therefore the vector $F(u)=(F_1(u),F_2(u))$ is constant on $O$ for any solution $u$ of the BO equation.

On the other hand, consider the following BO conservation laws
\begin{align*}
E_0(u) &=\int u^2\\
E_{1/2}(u) &=\int uH\dx u+\frac{1}{3}\int u^3.
\end{align*}
Set the following preserved vector
\begin{equation*}
E(u)=\left(E_0(u),E_{1/2}(u)\right).
\end{equation*}
$E(u)$ is in particular constant on $O^c$ for the solutions of $(\ref{BO}).$

Let $\mu_1$ and $\mu_2$ be two measures. We write $\mu_1\lhd\mu_2$ if there is a continuous increasing function $f$ vanishing at $0$ such that
\begin{equation*}
\mu_1(.)\leq f(\mu_2(.)).
\end{equation*}
This implies the absolute continuity of $\mu_1$ w.r.t. $\mu_2.$
For $\nu$ a probability measure on $H^2$, we define

\begin{align*}
\nu^O(.)&=\nu(.\cap O),\\
\nu^{O^c}(.)&=\nu(.\cap O^c),
\end{align*} 
where $O$ is the set described before.
\begin{prop}\label{Vect_abs_cont}
Suppose $\lambda_m\neq 0$ for all $m\in\Z$, then
\begin{enumerate}
\item \begin{equation*}
F_*{\mu^O_\alpha}\lhd \ell_2, \ where\ F=(F_1,\ F_2),
\end{equation*}
\item 
\begin{equation*}
E_*{\mu^{O^c}_\alpha}\lhd \ell_2,\ where\ E=(E_0,\ E_{1/2}).
\end{equation*}
\end{enumerate}
The functions describing the absolute continuity do not depend on $\alpha$ and $\ell_2$ is the Lebesgue measure on $\R^2.$
\end{prop}

\begin{proof}[Proof of Theorem \ref{Hausd}]
Let $W$ be an open set of $H^2.$ Clearly
\begin{equation*}
W=(W\cap O)\cup (W\backslash O).
\end{equation*}
By Proposition $\ref{Vect_abs_cont},$ we have
\begin{equation*}
\mu_\alpha(W)\leq f(\ell_2(F(W\cap O)))+g(\ell_2(E(W\backslash O))),
\end{equation*}
where $f$ and $g$ are the functions describing the absolute continuity established in Proposition $\ref{Vect_abs_cont}.$
Using the Portmanteau theorem, we get
\begin{equation}\label{HAUSSD}
\mu(W)\leq f(\ell_2(F(W\cap O)))+g(\ell_2(E(W\backslash O))),
\end{equation}
and by the regularity of $\mu$ and $\ell_2$ the estimate $(\ref{HAUSSD})$ holds for any bounded Borelian set $W$.

When $W$ is a compact set of Hausdorff dimension $\mathcal{H}(W)<2.$ It is clear that $E$ and $F$ are Lipschitz on any compact set. Since the Lipschitz maps do not increase the Hausdorff dimension, we have the right hand side of $(\ref{HAUSSD})$ is equal to $0$, then so is the left hand side.
\end{proof}

\begin{proof}[Proof of Proposition $\ref{Vect_abs_cont}$]
The proof consists of two steps:

\paragraph{$1.$ \textbf{Absolute continuity uniformly in $\alpha$ of $\mu$ on the set $O:$}}

 The first and second derivatives of the functionals $F_j(u)$ are
\begin{align*}
F_j'(u,v) &=\int u^jv,\\
F_j''(u,v) &=j\int u^{j-1}v^2.
\end{align*}
Then applying the Itô formula to $F_j,$ we find
\begin{equation*}
F_j(u)=F_j(u(0))+\int_0^tA_j(s)ds+\sqrt{\alpha}\sum_{m\in\Z}\lambda_m\int_0^t(u^j,e_m)d\beta_m(s),\ \ j=1,2.
\end{equation*}
where
\begin{equation*}
A_j=-(u^j,H\dy u-\alpha\dy u)+j\frac{\alpha}{2}\sum_{m\in\Z}\lambda_m^2(u^{j-1},e_m^2).
\end{equation*}
On the set $O$, we have $(u^j,H\dy u)=0$ $j=1,2$. Then recalling estimate on $\E\|u\|_2^2$ (Theorem $\ref{thm_inv_mes}$), we get
\begin{equation}\label{haus_est_drift}
\E|A_j|\leq \alpha\text{Const},
\end{equation}
where Const does not depend on $\alpha.$

We consider the $2\times2$-matrix $\sigma(u),\ u\in O$ with entries
\begin{equation*}
\sigma_{k,l}(u)=\sum_{m\in\Z}\lambda_m^2(u^k,e_m)(u^l,e_m), \ \ k,l=1,2.
\end{equation*}
It is clear that $\sigma$ is non-negative.
It follows from the Krylov estimate $(\ref{Krylov-est})$ with the use of the function $\Bbb 1_\Gamma$, $\Gamma$ being a Borel set of $\R^2,$ that
\begin{equation}\label{Krylov}
\E\left[(\det(\sigma(u)))^{1/2}\Bbb 1_\Gamma(F)\right] \leq C\ell_2(\Gamma),
\end{equation}
where $\ell_2$ is the Lebesgue measure on $\R^2$ and $C$ does not depend on $\alpha.$

Now define the map
\begin{equation*}
\fonction{D:\dot{H}^1(\T)}{\R_+}{u}{\det(\sigma(u))}.
\end{equation*}
We remark that $D$ is continuous since it is the composition of continuous maps. 
\begin{nem}\label{definiii}
Suppose $\lambda_m\neq 0$ for all $m\in\Z,$ then
\begin{equation*}
D(u)=0\Rightarrow u\equiv 0.
\end{equation*}
\end{nem}
\begin{proof}
Suppose there is a nonzero vector $\gamma=(\gamma_1,\gamma_2)\in\R^2$ such that
\begin{align*}
&\gamma\sigma(u)\gamma^T=0,
\end{align*}
then
\begin{align*}
\sum_{m\in\Z}\lambda_m^2\left(\sum_{j=1}^2\gamma_j(u^j,e_m)\right)^2=0.
\end{align*}
Since $\lambda_m\neq 0$ for all $m\neq 0,$ we infer that
\begin{equation*}
\sum_{j=1}^2\gamma_ju^j\equiv\text{Const},
\end{equation*}
which is possible only if $u\equiv 0,$ taking into account that $\int u=0.$ 
\end{proof}

Now define the set
\begin{equation*}
J_\epsilon =\left\{\|u\|_1^2\geq \epsilon,\ \ \|u\|_2^2\leq \frac{1}{\epsilon}\right\}\subset H^2(\T).
\end{equation*}
$J_\epsilon\cap O$ is a compact set in $H^1(\T)$ not containing $0$, then by the continuity of $D$, we have $D(J_\epsilon\cap O)$ is a compact set in $\R_+$ not containing $0$. Then there is $c_\epsilon>0$ such that $D(u)\geq c_\epsilon$ for all $u\in J_\epsilon \cap O.$ Using the same splitting argument as in the proof of Theorem $\ref{non_den_plus}$, we arrive at the claimed result.

\paragraph{$2.$ \textbf{Absolute continuity uniformly in $\alpha$ of $\mu$ on the set $O^c:$}}

We follow the construction above to set a $2\times2$-matrix $M$ with entries
\begin{equation*}
M_{k,l}(u)=\sum_{m\in\Z}\lambda_m^2B_k(u)B_l(u), \ \ k,l=1,2,
\end{equation*}
where
$$B_1=E_0'(u,e_m) \ \ \ \text{and}\ \ \ B_2=E'_{1/2}(u,e_m).$$
It follows from the Krylov estimate $(\ref{Krylov-est})$ that
\begin{equation*}
\E\left[(\det(M(u)))^{1/2}\Bbb 1_\Gamma(E)\right] \leq C\ell_2(\Gamma),
\end{equation*}
where $C$ does not depend on $\alpha$ thanks to the preservation of $E_0$ and $E_{1/2}$ by the BO flow.

Now $\det M(u)=0$ only if there is a nonzero vector $(\gamma_1,\gamma_2)\in\R^2$ such that

\begin{equation*}
\gamma_1 u+\gamma_2(2H\dx u+u^2)\equiv \text{Const}.
\end{equation*}
Note that if $\gamma_2=0$ we have that $u\equiv 0$ since $\int u=0$, therefore $u\in O$. Now we suppose that $\gamma_2\neq 0$, we differentiate w.r.t. $x$ to find
\begin{equation*}
\gamma_1\dx u+\gamma_2(2H\dy u+2u\dx u)\equiv 0.
\end{equation*}
Therefore, multiplying by $u^p$ for $p>0$ and integrating in $x$, we find
\begin{equation*}
\int u^pH\dy u=0,
\end{equation*}
and in particular $u$ belongs to the set $O$. Then on $O^c$ we have $\det(M(u))\neq 0.$ We can follow the same splitting argument with the use of the splitting set $J_\epsilon$ defined in the first part to get the result.
\end{proof}
\subsection*{A Gaussian decay property for the measure $\mu$}
Here we establish a large deviation bound for the measure $\mu$.
\begin{thm}\label{thm_gauss_dec}
The measure $\mu$ constructed in Theorem $\ref{thm_inv_mes}$ satisfies 
\begin{equation}
\E e^{\sigma\|u\|^2}<\infty,\label{gauss_control}
\end{equation}
where $\sigma=(aeA_0)^{-1}$ for arbitrary $a>1$.
In particular, for any $r>0$
\begin{equation*}
\mu(\{u\in C^\infty:\ \|u\|>r\})\leq C e^{-\sigma r^2},
\end{equation*}
where the constant $C$ does not depend on $r$.
\end{thm}
\begin{proof}
Recall the estimate $(\ref{estim_esp_lp})$:
\begin{align*}
\E\|u\|^{2p}\leq p^pA_0^p,
\end{align*}
then
\begin{equation*}
\E(\sqrt{\sigma}\|u\|)^{2p}\leq \sigma^p p^pA_0^p=\frac{p^p}{a^pe^p}.
\end{equation*}
Now, with use of the Stirling formula, we have
\begin{equation*}
\frac{\E(\sqrt{\sigma}\|u\|)^{2p}}{p!}\leq \frac{p^p}{p!a^pe^p}\sim_{p\to\infty} \frac{1}{a^p\sqrt{2\pi p}}.
\end{equation*}
Since $a>1$, we have that the series $$\sum_{p\geq 1}\frac{\E(\sqrt{\sigma}\|u\|)^{2p}}{p!}$$
is convergent, and we are led to $(\ref{gauss_control}).$ The other claim is obtained after combining $(\ref{gauss_control})$ with the Chebyshev inequality.
\end{proof}

\begin{rmq}
We obtain in a same way the result of Theorem $\ref{thm_gauss_dec}$ for the viscous measures uniformly in $\alpha.$

\end{rmq}

\section{Appendix}

\subsection*{Proof of Lemma $\ref{estim_123}$}
Note first that for a solution $v$ of the nonlinear equation $(\ref{equ_nonlin}),$ we have
\begin{equation}
\dt E_n(v)=E_n'(v,\dt v)=\alpha E'_n(v,\dy v)-\sqrt{\alpha}E_n'(v,\dx (vz))-\alpha\frac{1}{2}E_n'(v,\dx(z^2)),\ n=0,1,2.\label{Determinist_Ito!!}
\end{equation}
The $E_n$ are the first three conservation laws of the BO equation.
\paragraph{The case $n=0:$}
 $E_0'(v,w)=2\int vw.$ Applying $(\ref{Determinist_Ito!!})$, we get

\begin{align*}
\dt E_0(v)+2\alpha\|v\|_1^2 &=2\sqrt{\alpha}(v,\dx (vz))+\alpha (v,\dx z^2)\\
&=\sqrt{\alpha}(v^2,\dx z)+\alpha(v,\dx z^2)\\
&\leq \sqrt{\alpha}\|z\|_{\frac{3+}{2}}\|v\|^2+c\alpha\|v\|\|z\|_1^2\\
&\leq \sqrt{\alpha}\|z\|_{\frac{3+}{2}}\|v\|^2+c\alpha(1+\|v\|^2)\|z\|_1^2.
\end{align*}

Note that $\|z(.)\|_{\frac{3+}{2}}$ is bounded uniformly in $\alpha$  for almost all realizations and in $t$ (on $[0,T]$) by continuity, then with the use of the Gronwall inequality we get

\begin{equation}\label{e1}
\sup_{t\in[0,T]}\|v(t)\|^2+2\alpha\int_0^T\|v(t)\|_1^2dt\leq C(T,\omega,\|v_0\|).
\end{equation}

\paragraph{The case $n=1:$}
Recall that
\begin{align*}
E_1(u)&=\int (\dx u)^2+\frac{3}{4}\int u^2H\dx u+\frac{1}{8}\int u^4.
\end{align*}
Then
\begin{equation*}
E_1'(v,w)=-2(\dy v,w)+\underbrace{\frac{3}{2}(vH\dx v,w)+\frac{3}{4}(v^2,H\dx w)+\frac{1}{2}(v^3,w)}_{R_1'(v,w)}.
\end{equation*}
It was already shown that (see the more general estimates $(\ref{estim_r_n_2,j})$ and $(\ref{estim_R_n^1})$)

\begin{equation*}
|R_1'(v,\dy v)|\leq \epsilon\|v\|_2^2+C_\epsilon\|v\|^c.
\end{equation*}
Then

\begin{equation*}
\alpha E_1'(v,\dy v)\leq -(2-\epsilon)\alpha\|v\|_2^2+C_\epsilon\alpha\|v\|^c.
\end{equation*}
Taking into account some properties of $H$, it suffices to treat $(vH\dx v,w)+(v^3,w)$ instead of $R_1'(v,w)$ for our purpose.\\ Now
\begin{align*}
\sqrt{\alpha}|(\dy v, \dx (vz))| &\leq C\sqrt{\alpha}\|v\|_2\|v\|_1\|z\|_1\\
&\leq \epsilon\alpha\|v\|_2^2+C_\epsilon\|v\|_1^2\|z\|_1^2\leq \epsilon\alpha\|v\|_2^2+C_{T,\epsilon,\omega}\|v\|_1^2,\\
\sqrt{\alpha}|(vH\dx v,\dx (vz))| &= \sqrt{\alpha}|(\dx(vH\dx v),vz)|\\
&\leq C\sqrt{\alpha}\|v\|_1\|v\|_2\|v\|\|z\|_{\frac{1+}{2}}\\
&\leq \epsilon\alpha\|v\|_2^2+C_\epsilon\|v\|_1^2\|v\|^2\|z\|_{\frac{1+}{2}}^2\leq \epsilon\alpha\|v\|_2^2+C_{T,\epsilon,\omega}\|v\|_1^2,\\
\sqrt{\alpha}|(v^3,\dx (vz))|&\leq C\sqrt{\alpha}\|v\|_{L^6}^3\|v\|_1\|z\|_1\\
&\leq C\sqrt{\alpha}\|v\|_{1/3}^3\|v\|_1\|z\|_1\\
&\leq C\sqrt{\alpha}\|v\|^2\|v\|_1^2\|z\|_1\leq \sqrt{\alpha}C_{T,\omega}\|v\|_1^2.
\end{align*}
To summarize, we have
\begin{equation*}
\sqrt{\alpha}E_1'(v,\dx (vz))\leq \epsilon\|v\|_2^2+C_{T,\epsilon,\omega}\|v\|_1^2.
\end{equation*}
To estimate the last term, we compute
\begin{align*}
\alpha|(\dy v,\dx z^2)| &\leq C \alpha\|v\|_2\|z\|_1^2\\
&\leq \epsilon\alpha\|v\|_2^2+\alpha C_\epsilon\|z\|_1^4,\\
\alpha|(vH\dx v,\dx z^2)| &\leq C\alpha\|v\|_2\|v\|_1\|z\|_{1/4}^2\\
&\leq \epsilon\alpha\|v\|_2^2+\alpha C_{T,\epsilon,\omega}\|v\|_1^2,\\
\alpha|(v^3,\dx z^2)|&\leq C\alpha\|v\|^2\|v\|_1\|z\|_1^2\\
&\leq\epsilon\alpha\|v\|_1^2+\alpha C_{T,\epsilon,\omega}.
\end{align*}
To conclude, we can choose $\epsilon$ so that
\begin{equation*}
E_1(v)+\alpha\int_0^t\|v(r)\|_2^2dr\leq E_1(v_0)+C^1_{T,\omega}\int_0^t\|v(r)\|_1^2dr+C^2_{T,\omega}t.
\end{equation*}
Recalling the inequality $(\ref{apriori_est_on_En})$ and $(\ref{e1})$ we have 
\begin{equation*}
\|v\|_1^2+2\alpha\int_0^t\|v(r)\|_2^2dr\leq E_1(v_0)+C^0_{T,\omega}+C^1_{T,\omega}\int_0^t\|v(r)\|_1^2dr+C^2_{T,\omega}t.
\end{equation*}

With the use of the Gronwall lemma, we arrive at
\begin{equation*}
\sup_{t\in[0,T]}\|v(t)\|_1^2+2\alpha\int_0^T\|v(t)\|_2^2dt\leq C_{T,\omega}(\|v_0\|_1).
\end{equation*}

\paragraph{The case $n=2:$}
Recall that
\begin{align*}
E_2(u) &= \int (\dy u)^2-\frac{5}{4}\int \left((\dx u)^2H\dx u+2\dy uH\dx u\right)\\
&+\frac{5}{16}\int\left(5u^2(\dx u)^2+u^2(H\dx u)^2+2uH(\dx u)H(u\dx u)\right)\\
&+\int\left(\frac{5}{32}u^4H(\dx u)+\frac{5}{24}u^3H(u\dx u)\right)+\frac{1}{48}\int u^6.
\end{align*}
The form of $E_2(v)$ combined with some properties of $H$ allows us to reduce to the treatment of the quantity
\begin{equation*}
R_2(v)=\|v\|_2^2+\int(\dx v)^3+(\dy v,H\dx v)+(v^2,(\dx v)^2)+(v^4,H\dx v)+\int v^6.
\end{equation*}
Then 
\begin{align*}
R_2'(v,w)&=2(\dy v,\dy w)+3((\dx v)^2,\dx w)+2(\dy v,H\dx w)+2(vw,(\dx v)^2)\\ &+2(v^2\dx v,\dx w)+4(v^3H\dx v,w)+(v^4,H\dx w)+6(v^5,w)\\
&=2(\dy v,\dy w)+R_3'(v,w).
\end{align*}
It was already shown in the proof Lemma $\ref{lemme_central}$ (see estimates $(\ref{estim_r_n_2,j})$ and $(\ref{estim_R_n^1})$) that
\begin{equation*}
|R_3'(v,\dy v)|\leq \epsilon\|v\|_3^2+C_\epsilon\|v\|^c,
\end{equation*}
for some constants $c,\ C_\epsilon>0$.
Now we have
\begin{align*}
2\alpha(\dy v,\dy(\dy v))=-2\alpha\|v\|_3^2.
\end{align*}
Then
\begin{equation*}
\alpha E_2'(v,\dy v)\leq -(2-\epsilon)\alpha\|v\|_3^2+\alpha C_\epsilon\|v\|^c.
\end{equation*}
Now
\begin{align*}
\sqrt{\alpha}|(\dy v, \dy (\dx(vz)))| &\leq C\sqrt{\alpha}\|v\|_3\|v\|_2\|z\|_2\leq \epsilon\alpha\|v\|_3^2+C_{T,\epsilon,\omega}\|v\|_2^2,\\
\sqrt{\alpha}|((\dx v)^2,\dy (vz))| &\leq C_{T,\omega}\sqrt{\alpha}\|v\|_{5/4}^2\|v\|_2\\
&\leq C_{T,\omega}\sqrt{\alpha}\|v\|_{1/4}\|v\|_{2}^2\leq C_{T,\omega}\sqrt{\alpha}\|v\|_2^2,\\
\sqrt{\alpha}|(\dy (vz),H\dy v)| &\leq\sqrt{\alpha}C\|v\|_2^2\|z\|_2\leq \sqrt{\alpha} C_{T,\omega}\|v\|_2^2,\\
\sqrt{\alpha}|(v(\dx v)^2,\dx (vz))| &\leq C\sqrt{\alpha}\|v\|\|z\|_{\frac{1+}{2}}\|v\|_2^2\leq C_{T,\omega}\sqrt{\alpha}\|v\|_2^2,\\
\sqrt{\alpha}|(v^2\dx v,\dy (vz))| &\leq \sqrt{\alpha} C\|v\|_1^3\|v\|_2\|z\|_2\leq C_{T,\omega}\sqrt{\alpha}\|v\|_2^2,\\
\sqrt{\alpha}|(v^3H\dx v,\dx (vz))| &\leq C\sqrt{\alpha} \|v\|_1^3\|v\|_2\|z\|_{\frac{1+}{2}}\|v\|\leq \sqrt{\alpha}C_{T,\omega}\|v\|_2^2,\\
\sqrt{\alpha}|(v^4,H\dy (vz))| &\leq\sqrt{\alpha} C\|v\|_1^5\|z\|_1\leq \sqrt{\alpha}C_{T,\omega},\\
\sqrt{\alpha}|(v^5,\dx (vz))| &\leq C\sqrt{\alpha}\|v\|_1^5\|v\|\|z\|_{\frac{1+}{2}}\leq \sqrt{\alpha}C_{T,\omega}.
\end{align*}
The estimates concerning the term $\dx z^2$ are easier because they do not contain $v$. Finally, using the same argument as before (in the case of $E_1(v)$), we arrive at the claimed result.

\subsection*{The periodic Hilbert transform}

We present in this section a definition of the Hilbert transform in the periodic setting and establish some of its elementary properties. 
Recall that the sequence defined by
$$e_n(x)=\left\{\begin{array}{l r c} 
\frac{\sin(nx)}{\sqrt{\pi}}\ \ \ \ \text{if}\ n<0,\\
\frac{\cos(nx)}{\sqrt{\pi}}\ \ \ \ \text{if}\ n>0,
\end{array}
\right.$$
forms a Hilbertian basis of $\dot{H}(\T)$, let us denote this basis by $\mathcal{B}$. We define the Hilbert transform on $\mathcal{B}$ by
 \begin{equation*}
 He_n(x)=sgn(n)e_{-n}(x),
 \end{equation*}
 where 
 $$sgn(p)=\left\{\begin{array}{l r c}
 1\ \ \ \ \text{if}\ p>0,\\
 0\ \ \ \ \text{if}\ p=0,\\
 -1\ \ \ \ \text{if}\ p<0.
 \end{array}
 \right.$$
 We first remark that $H$ defines an isometry on $\dot{H}$.
 \begin{prop}\label{Prop_de_H}
 Let $f,g\in \dot{H}(\T)$, then
 \begin{align}
 H^2f &=-f\label{H_involution}\\
 \int_\T Hf &=0\label{H_int_null}\\
 (g,Hf) &=-(Hg,f)\label{H_anti_sym}\\
 \widehat{Hf}_0(p) &=-i\osgn(p)\hat{f}_0(p)\label{H_def_by_Fourier},
 \end{align}
 where $\hat{h}_0$ denotes the complex Fourier coefficient of a function $h$ defined below.
 \end{prop}
 Define now the Fourier coefficients associated to a function $f$ in $\dot{H}$:
  \begin{align*}
 \hat{f}_1(n) &=\frac{1}{\sqrt{\pi}}\int_\T \cos(nx)f(x)dx\\
 \hat{f}_2(n) &=\frac{1}{\sqrt{\pi}}\int_\T \sin(nx)f(x)dx.
 \end{align*}
The function $f$ is represented  in $\mathcal{B}$ as follows:
  \begin{equation}
 f(x)=\sum_{n>0}(\hat{f}_1(n)e_n(x)-\hat{f}_2(n)e_{-n}(x)).\label{H_repr_f_Four}
 \end{equation}
 Hence the Hilbert transform of $f$ can be expressed as
 \begin{equation}
 Hf(x)=\sum_{n>0}(\hat{f}_1(n)e_{-n}(x)+\hat{f}_2(n)e_{n}(x)).\label{H_repr_Hf_Four}
 \end{equation}
 The complex Fourier coefficient is defined by
 \begin{equation}
 \hat{f}_0(p)=\frac{1}{\sqrt{2\pi}}\int_\T e^{-ipx}f(x)dx.\label{H_complex_coeficient_Four}
 \end{equation}
 The relationship between the three Fourier coefficients of $f$ is
 \begin{equation}
 \hat{f}_0(p)=\frac{\hat{f}_1(p)-i\osgn(p)\hat{f}_2(p)}{\sqrt{2}}.\label{H_relat_bet_3_ceoff}
 \end{equation}
 
\begin{proof}[Proof of Proposition $\ref{Prop_de_H}$]

Equation $(\ref{H_int_null})$ follows immediately from $(\ref{H_repr_Hf_Four})$. 
 Now from $(\ref{H_repr_f_Four})$ and $(\ref{H_repr_Hf_Four})$, we can easily deduce that
 \begin{align*}
 H^2f(x) &=-\sum_{n>0}(\hat{f}_1(n)e_{n}(x)-\hat{f}_2(n)e_{-n}(x))=-f(x).
 \end{align*}
and $(\ref{H_involution})$ is shown.
 
 From $(\ref{H_repr_Hf_Four})$, we infer that
 \begin{align*}
 \widehat{Hf}_1(p) =-\hat{f}_2(p),\ \ \ \ \widehat{Hf}_2(p) =\hat{f}_1(p).
 \end{align*}
 Thus using the relation $(\ref{H_relat_bet_3_ceoff})$, we can write
 \begin{align*}
 \widehat{Hf}_0(p)&=\frac{-\hat{f}_2(p)-i\osgn(p)\hat{f}_1(p)}{\sqrt{2}}\\
 &=\frac{-i\osgn(p)(\hat{f}_1(p)-i\osgn(p)\hat{f}_2(p))}{\sqrt{2}}\\
 &=-i\osgn(p)\hat{f}_0(p),
 \end{align*}
 and we have arrived at $(\ref{H_def_by_Fourier})$.
 
  To prove $(\ref{H_anti_sym})$, we compute
 \begin{align*}
 (g,Hf) &=\sum_{n>0}\hat{f}_1(n)\int_\T g(x)e_{-n}(x)dx+\sum_{n>0}\hat{f}_2(n)\int_\T g(x)e_{n}(x)dx\\
 &=-\sum_{n>0}\hat{f}_1(n)\hat{g}_2(n)+\sum_{n>0}\hat{g}_1(n)\hat{f}_2(n)\\
 &=-\sum_{n>0}\hat{g}_2(n)\int_\T f(x)e_n(x)dx-\sum_{n>0}\hat{g}_1(n)\int_\T f(x)e_{-n}(x)dx\\
 &=-\int_\T f(x)\sum_{n>0}(\hat{g}_1(n)e_{-n}(x)+\hat{g}_2(n)e_n(x))\\
 &=-(Hg,f).
 \end{align*}
 \end{proof}

\paragraph{Acknowledgements.} I thank my advisors Armen Shirikyan and Nikolay Tzvetkov for useful discussions and valuable remarks.  I am grateful to the referee for valuable remarks that were very useful for improving the text. This research was supported by the program DIM RDMath of "Région Ile-de-France".
 
 \bibliographystyle{alpha}
\bibliography{boob}

\begin{thebibliography}{ABFS89}

\bibitem[ABFS89]{abdelouhab1989nonlocal}
L.~Abdelouhab, J.~L. Bona, M.~Felland, and J.C. Saut.
\newblock Nonlocal models for nonlinear, dispersive waves.
\newblock {\em Physica D: Nonlinear Phenomena}, 40(3):360--392, 1989.

\bibitem[Ben67]{benjamin1967internal}
T.~B. Benjamin.
\newblock Internal waves of permanent form in fluids of great depth.
\newblock {\em Journal of Fluid Mechanics}, 29(03):559--592, 1967.

\bibitem[Den15]{yudeng}
Y.~Deng.
\newblock Invariance of the {G}ibbs measure for the {B}enjamin-{O}no equation.
\newblock {\em J. Eur. Math. Soc. (JEMS)}, 2015.

\bibitem[DTV15]{dtv2014invariant}
Y.~Deng, N.~Tzvetkov, and N.~Visciglia.
\newblock Invariant measures and long time behaviour for the {B}enjamin-{O}no
  equation {III}.
\newblock {\em Comm. Math. Phys.}, 339(3):815--857, 2015.

\bibitem[Kry87]{kry}
N.~V. Krylov.
\newblock On estimates of the maximum of a solution of a parabolic equation and
  estimates of the distribution of a semimartingale.
\newblock {\em Mathematics of the USSR-Sbornik}, 58(1):207, 1987.

\bibitem[KS91]{karatshre}
I.~Karatzas and S.~E. Shreve.
\newblock {\em Brownian motion and stochastic calculus}.
\newblock Graduate texts in Mathematics, {S}pringer-{V}erlag, New York, 1991.

\bibitem[KS04]{KS04}
S.~Kuksin and A.~Shirikyan.
\newblock Randomly forced {CGL} equation: stationary measures and the inviscid
  limit.
\newblock {\em Journal of Physics. A. Mathematical and General}, 37:3805--3822,
  2004.

\bibitem[KS12]{KS12}
S.~Kuksin and A.~Shirikyan.
\newblock {\em Mathematics of Two-Dimensional Turbulence}.
\newblock Cambridge University Press, Cambridge, 2012.

\bibitem[Kuk04]{kuk_eul_lim}
S.~Kuksin.
\newblock The {E}ulerian limit for 2{D} statistical hydrodynamics.
\newblock {\em J. Statist. Phys.}, 115(1-2):469--492, 2004.

\bibitem[Kuk08]{kuksin_nondegeul}
{S.} Kuksin.
\newblock On distribution of energy and vorticity for solutions of 2d
  {N}avier-{S}tokes equation with small viscosity.
\newblock {\em Communications in Mathematical Physics}, 284(2):407--424, 2008.

\bibitem[Mat84]{matsuno1984bilinear}
Y.~Matsuno.
\newblock {\em Bilinear transformation method}.
\newblock Mathematics in Science and Engineering. Academic Press, Inc.,
  Orlando, FL, 1984.

\bibitem[Mol08]{molinet2008global}
L.~Molinet.
\newblock Global well-posedness in {$L^2$} for the periodic {B}enjamin-{O}no
  equation.
\newblock {\em American Journal of Mathematics}, 130(3):635--683, 2008.

\bibitem[MP12]{molinet2012cauchy}
L.~Molinet and D.~Pilod.
\newblock The {C}auchy problem for the {B}enjamin--{O}no equation in {$L^2$}
  revisited.
\newblock {\em Analysis \& PDE}, 5(2):365--395, 2012.

\bibitem[Ono75]{ono}
H.~Ono.
\newblock Algebraic solitary waves in stratified fluids.
\newblock {\em Journal of the Physical Society of Japan}, 39(4):1082--1091,
  1975.

\bibitem[Shi11]{armen_nondegcgl}
A.~Shirikyan.
\newblock Local times for solutions of the complex {G}inzburg-{L}andau equation
  and the inviscid limit.
\newblock {\em J. Math. Anal. Appl.}, 384(1):130--137, 2011.

\bibitem[TV13]{tv2011gaussian}
N.~Tzvetkov and N.~Visciglia.
\newblock Gaussian measures associated to the higher order conservation laws of
  the {B}enjamin-{O}no equation.
\newblock {\em Ann. Sci. \'Ec. Norm. Sup\'er. (4)}, 46(2):249--299, 2013.

\bibitem[TV14]{tzvetkov-visc}
N.~Tzvetkov and N.~Visciglia.
\newblock Invariant measures and long-time behavior for the {B}enjamin-{O}no
  equation.
\newblock {\em Int. Math. Res. Not. IMRN}, (17):4679--4714, 2014.

\bibitem[TV15]{tv2014invariant}
N.~Tzvetkov and N.~Visciglia.
\newblock Invariant measures and long time behaviour for the {B}enjamin-{O}no
  equation {II}.
\newblock {\em J. Math. Pures Appl. (9)}, 103(1):102--141, 2015.

\bibitem[Tzv06]{tzvetkov2004ill}
N.~Tzvetkov.
\newblock Ill-posedness issues for nonlinear dispersive equations.
\newblock In {\em Lectures on nonlinear dispersive equations}, volume~27 of
  {\em Gakuto Internat. Ser. Math. Sci. Appl.}, pages 63--103. 2006.

\bibitem[Tzv10]{tzvetkov2010construction}
N.~Tzvetkov.
\newblock Construction of a {G}ibbs measure associated to the periodic
  {B}enjamin--{O}no equation.
\newblock {\em Probability theory and related fields}, 146(3-4):481--514, 2010.

\bibitem[Zhi01a]{zhidkov2001korteweg}
P.~E. Zhidkov.
\newblock {\em Korteweg-de Vries and nonlinear Schr{\"o}dinger equations:
  qualitative theory}.
\newblock Number 1756. Springer Science \& Business Media, 2001.

\bibitem[Zhi01b]{zhid_shrod}
P.~E. Zhidkov.
\newblock On an infinite sequence of invariant measures for the cubic nonlinear
  {S}chr\"odinger equation.
\newblock {\em Int. J. Math. Math. Sci.}, 28(7):375--394, 2001.

\end{thebibliography}
\end{document}